\documentclass[11pt]{amsart}
\usepackage{amsmath, amssymb, amsthm, enumerate, graphicx,color,soul,verbatim,algorithmic}
\usepackage[dvipsnames]{xcolor}
\usepackage[colorlinks, linkcolor=MidnightBlue, citecolor=magenta, urlcolor=MidnightBlue
]{hyperref}
\usepackage[nameinlink]{cleveref}
\usepackage[foot]{amsaddr}
\usepackage{relsize}
\usepackage[utf8]{inputenc}
\usepackage{float}
\usepackage{setspace}
\usepackage{mathpazo}
\usepackage{indentfirst}  
\usepackage{fullpage}
\usepackage{thmtools}
\usepackage{thm-restate}
\usepackage{tikz}
\usepackage{svg}
\usepackage{bm}
\usepackage[symbol]{footmisc}
\usepackage[backend=bibtex,style=alphabetic,maxbibnames=99]{biblatex}
\usetikzlibrary{arrows.meta}

\numberwithin{equation}{section}

\newcommand{\ZZ}{\mathbb{Z}}
\newcommand{\RR}{\mathbb{R}}
\newcommand{\Q}{\mathbb{Q}}
\newcommand{\Pol}{\mathcal P{ol}}
\newcommand{\bv}{\mathbf v}
\newcommand{\be}{\mathbf e}
\newcommand{\ba}{\mathbf a}
\newcommand{\bc}{\mathbf c}
\newcommand{\bx}{\mathbf x}
\newcommand{\bu}{\mathbf u}

\newtheorem{theorem}{Theorem}[section]

\newtheorem{result}[theorem]{Result}

\newtheorem{prop}[theorem]{Proposition}

\theoremstyle{definition}
\newtheorem{definition}[theorem]{Definition}
\newtheorem{algorithm}[theorem]{Algorithm}

\newtheorem{example}[theorem]{Example}

\newtheorem{question}[theorem]{Question}

\newtheorem{remark}[theorem]{Remark}

\newcommand{\by}{{\bf y}}
\newcommand{\bz}{{\bf z}}

\newcommand{\TP}{\mathbb{TP}}

\newcommand{\conv}{\text{conv}}
\newcommand{\tconv}{\text{tconv}}

\newcommand{\R}{{\mathbb R}}
\newcommand{\Z}{{\mathbb Z}}

\DeclareMathOperator{\ehr}{ehr}
\DeclareMathOperator{\Ehr}{Ehr}
\DeclareMathOperator{\Vol}{Vol}
\DeclareMathOperator{\vol}{vol}
\DeclareMathOperator{\Todd}{Todd}

\newcommand{\Qa}[1][a]{Q(\mathbf{#1})}

\author{Marie-Charlotte Brandenburg$^1$}
\address{\hspace*{-1.1em}$^1$Max Planck Institute for Mathematics in the Sciences, Leipzig, Germany}
\author{Sophia Elia$^2$}
\address{$^2$Freie Universität Berlin, Berlin, Germany}
\author{Leon Zhang$^3$}
\address{$^3$University of California, Berkeley, USA}
\title{Multivariate volume, Ehrhart, and $h^*$-polynomials of polytropes}
\email{marie.brandenburg@mis.mpg.de, sophiae56@zedat.fu-berlin.de, leonyz@berkeley.edu}

\begin{document}

\begin{abstract}
The univariate Ehrhart and $h^*$-polynomials of lattice polytopes have been widely studied. 
We describe methods from toric geometry for computing multivariate versions of volume, Ehrhart and $h^*$-polynomials of lattice polytropes, which are both tropically and classically convex, and are also known as alcoved polytopes of type $A$. These algorithms are applied to all polytropes of dimensions $2,3$ and $4$, yielding a large class of integer polynomials. We give a complete combinatorial description of the coefficients of volume polynomials of $3$-dimensional polytropes in terms of regular central subdivisions of the fundamental polytope. Finally, we provide a partial characterization of the analogous coefficients in dimension $4$.
\end{abstract}

\maketitle

\section{Introduction}

Polytropes are a fundamental class of polytopes, which masquerade in the literature as alcoved polytopes of type $A$ \cite{LP, LP2}.
Among many others, they include order polytopes, some associahedra and matroid polytopes, hypersimplices, and Lipschitz polytopes.
They are tropical polytopes which are classically convex \cite{JK} and 
are closely related to the notion of Kleene stars and the problem of finding shortest paths in weighted graphs \cite{T, JS}. 
Polytropes also arise in a range of algorithmic applications to other fields, including
phylogenetics \cite{YZZ}, mechanism design \cite{CT}, and building theory \cite{JSY}.

It is well known that computing and approximating the volume of a polytope is ``difficult'' \cite{barany_computingvolumeis}. More specifically, there is no polynomial-time algorithm for the exact computation of the volume of a polytope \cite{dyer_computingvolumeconvex}, even when restricting to the class of polytopes defined by a totally unimodular matrix.
However, viewing polytropes as the ``building blocks'' of tropical polytopes, understanding their volumes provides insight into the volume of tropical polytopes. Determining whether the volume of such a tropical polytope is zero is equivalent to deciding whether a mean payoff game is winning \cite{akian_tropicalpolyhedraare}. The volume of a tropical polytope can hence serve as a measurement of how far a game is from being winning \cite{GM}.

Unimodular triangulations of polytropes were studied in the language of affine Coxeter arrangements in \cite{LP}, producing a volume formula and non-negativity of the $h$-vector corresponding to the triangulation.
Motivated by a novel possibility for combining algebraic methods with enumerative results from tropical geometry, we continue to study the volume of polytropes, both continuously and discretely. 
The Ehrhart counting function encodes the discrete volume by counting the number of lattice points in any positive integral dilate of a polytope. For lattice polytopes, this counting function is given by a univariate polynomial, the Ehrhart polynomial, with leading term equal to the Euclidean volume of the polytope. Rewriting the Ehrhart polynomial in the basis of binomial coefficients determines the $h^*$-polynomial and reveals additional beautiful connections between the coefficients and the geometry of the polytope. 
 It is an area of active research to determine the relations between the $h^*$-coefficients of alcoved polytopes \cite[Question 1]{SchepersVanLang}; for example, it is conjectured that the $h^*$-vectors of alcoved polytopes of type $A$ are unimodal.

In recent work, Loho and Schymura \cite{LS} developed a separate notion of volume for tropical polytopes driven by a tropical version of dilation, which yields an Ehrhart theory for a new class of tropical lattices. This notion of volume is intrinsically tropical and exhibits many natural properties of a volume measure, such as being monotonic and rotation-invariant. Nevertheless, the discrete and classical volume can be more relevant for certain applications; for example, the irreducible components of a Mustafin variety correspond to the lattice points of a certain tropical polytope \cite{CHSW, Z}.

We pass from univariate polynomials to multivariate polynomials to push the connections between the combinatorics  of the polynomials and the geometry even further.
Combinatorial types of polytropes have been classified up to dimension $4$ \cite{T, JS}. 
Each polytrope of the same type has the same normal fan. 
Given a normal fan, we create multivariate polynomial functions in terms of the rays that yield the (discrete) volume and $h^*$-evaluation for any polytrope of that type. 
We first use algebraic methods to compute the multivariate volume polynomials, following the algorithm in \cite{DLS}. We then transform these polynomials into multivariate Ehrhart polynomials, which are highly related to vector partition functions, using the Todd operator. Finally we perform the change of basis to recover the $h^*$-polynomials. 

\begin{result}
We compute the multivariate volume, Ehrhart, and $h^*$-polynomials for all types of polytropes of dimension~$\leq 4$. 
\end{result}
Furthermore, these methods could be extended to higher dimensions with increased computation power.
Our code and the resulting polynomials are publicly available on a Github repository\footnote{\url{https://github.com/mariebrandenburg/polynomials-of-polytropes}
}.

Each combinatorial type of polytrope of dimension $n-1$ corresponds to a certain triangulation of the fundamental polytope $FP_{n}$, the polytope with vertices $e_i-e_j$ for $i,j\in [n]$ \cite{JS}. Our computations show that the volume polynomials of polytropes of dimension $3$ have integer coefficients with a strong combinatorial meaning:

\begin{theorem}
 The coefficients of the volume polynomials of maximal $3$-dimensional polytropes reflect the combinatorics of the corresponding regular central subdivision of $FP_3$.
\end{theorem}
For example, each coefficient of a monomial of the form $a_{ij}a_{kl}a_{st}$ is either $6$ or $0$. This reflects whether the vertices $e_i-e_j, e_k - e_l$ and $e_s-e_t$ form a face in the triangulation of $FP_4$ or not. Similarly, the coefficient of the monomial $a_{ij}^2a_{kl}$ is $-3$ if the vertex $e_k - e_l$ is incident to a triangulating edge of a square facet of $FP_3$ and $0$ otherwise. 
These intriguing observations naturally lead to a question of generalization.

\begin{question}
 How do the coefficients of the volume polynomials of maximal $(n-1)$-dimensional polytropes reflect the combinatorics of the corresponding regular central subdivision of $FP_n$?
\end{question}
To emphasize this question, we show that our data of volume polynomials of dimension $4$ is highly structured:

\begin{theorem}
In the 8855-dimensional space of homogeneous polynomials of degree 4, the 27248 normalized volume polynomials of 4-dimensional polytropes span a 70-dimensional affine subspace.
\end{theorem}

Finally, we present a partial characterization of the coefficients of these polynomials. For example, the coefficient of a monomial of the form $a_{ij}a_{ik}$ is always either $0$ or $6$, and the sum of all coefficients of this form is always $300$, in each of the $27248$ polynomials. 

 \subsection*{Overview}

In this article we describe methods for computing the multivariate volume, Ehrhart, and $h^*$-polynomials for all polytropes.  We begin by describing the Ehrhart theory, tropical geometry, and algebraic geometry necessary for these methods in \Cref{sec:background}. 
In \Cref{sec:methods}, we describe our methods and apply them to 2-dimensional polytropes.
In Section \ref{sec:compuatations}, we apply these methods to compute the volume, Ehrhart, and $h^*$-polynomials of polytropes of dimension 3 and 4. We give a complete description of the coefficients of volume polynomials of 3-dimensional polytropes in terms of regular central subdivisions of the fundamental polytope, and give a partial characterization of these coefficients in dimension $4$. 

\subsection*{Acknowledgements}

The authors thank Christian Haase, Michael Joswig, Benjamin Schr\"oter, Rainer Sinn, and Bernd Sturmfels for many helpful conversations, and the anonymous referees for input that greatly improved this article. They also thank Michael Joswig, Lars Kastner, and Benjamin Schr\"oter for providing a dataset for 4-dimensional polytropes, and the Max Planck Institute for Mathematics in the Sciences for its hospitality while working on this project. Leon Zhang was partially supported by an NSF Graduate Research Fellowship. Sophia Elia was supported by the Deutsche Forschungsgemeinschaft (DFG)
Graduiertenkolleg “Facets of Complexity”.

\section{Background}
\label{sec:background}

In this section we give a brief overview of the background material we use for our results. Note that throughout this article, we assume that $P$ is a lattice polytope unless stated otherwise.

\subsection{Ehrhart theory}
\label{subsec:ehrhart}
In this subsection, we recall some essentials of Ehrhart theory.
For more details, we refer the reader to \cite{BR}.

A \emph{lattice polytope} $P\subseteq \mathbb{R}^n$ is the convex hull of a finite set of points in $\mathbb{Z}^n$. 
Equivalently, lattice polytopes are bounded intersections of finitely many closed half-spaces:
$P = \{ \mathbf x \in \R ^n : A \mathbf x \leq \mathbf{a} \}$ for some 
$A \in \ZZ ^{m \times n}$ and $\mathbf a \in \ZZ ^m$.
The \emph{Ehrhart counting function} of $P$, written $\ehr_P(k)$, gives the number of lattice points in the $k$-th dilate of $P$ for $k \in \mathbb Z _{\geq 1}$: $$\ehr_P(k)= |kP \cap \mathbb Z ^n| = |\{ \mathbf x \in \ZZ ^n : A \mathbf x \leq k\mathbf{a} \}|.  $$
Ehrhart's theorem \cite{Ehrhart} says that for positive integers, $\ehr_P(k)$ agrees with a polynomial in $k$ of degree equal to the dimension of $P$. Furthermore, the constant term of this polynomial is equal to 1 and the coefficient of the leading term is equal to the Euclidean volume of $P$ within its affine span. The interpretation of other coefficients of the Ehrhart polynomial is an active direction of research.

Generating functions play a central role in Ehrhart theory. The \emph{Ehrhart series} $\Ehr_P(t)$ of a polytope $P$ is the formal power series given by $$\Ehr_P(t) = 1 + \sum_{k \geq 1} \ehr_P(k)t^k. $$ 
For a $d$-dimensional lattice polytope, the Ehrhart series has the rational expression
$$\Ehr_P(t)=1 + \sum_{k \geq 1}\ehr_P(k)t^k = \frac {h^\ast_P(t)}{(1-t)^{d+1}} ,$$
where $h^*_P(t) = \sum_{i=0}^d h_i^* t ^i $ is a polynomial in $t$ of degree at most $d$, called the \emph{$h^*$-polynomial}. Furthermore, each $h_i^*$ is a non-negative integer \cite{stanley1980decompositions}. The coefficients of the $h^*$-polynomial form the \emph{$h^*$-vector}: $(h^*_0,h^*_1,\dots, h_d^*)$.

The \emph{normalized volume} of $P$ is defined as $\Vol(P)=\dim(P)! \vol(P),$ where $\vol(P)$ is the Euclidean volume of $P$ within its affine span. It is equal to the sum of the coefficients of the $h^*$-polynomial. The Ehrhart polynomial may be recovered from the $h^\ast$-vector through the transformation
$$\ehr_P(k) = \sum_{i=0}^{d}h_i^*  \binom{k+d -i}{d}.$$

For a lattice polytope $P = \{ x \in \RR^n : A\bx \leq \mathbf b\}$ with $A \in \ZZ^{m \times n},\, \mathbf b \in {\ZZ^m}$,
the \emph{multivariate Ehrhart counting function} of $P$, $\ehr_P(\ba) :\ZZ^m \rightarrow \ZZ$, gives the number of lattice points in the vector dilated polytope: 
$$\ehr_P(\ba) = |\{ x \in \ZZ^n : A\bx \leq \mathbf \ba\} |.$$
This counting function is closely related to vector partition functions, which can be used to show that $\ehr_P(\ba)$ is piecewise-polynomial \cite{DM1988}.
Vector partition functions and the related Ehrhart theory have been widely studied, see, for example \cite{vector_partition_functions,HenkLinke}.

\subsection{Tropical convexity}
In this subsection we review some basics of tropical arithmetic and tropical convexity. We refer readers to \cite{DS} or \cite{J2} for a more detailed exposition.

Over the min-plus tropical semiring $\mathbb T = (\mathbb R\cup \{\infty\}, \oplus, \odot)$ we define for $a,b\in \mathbb T$ the operations of addition $a\oplus b$ and multiplication $a\odot b$ by
$$a \oplus b = \min(a, b),\ \  a \odot b = a + b.$$
We can similarly define vector addition and scalar multiplication: for any scalars $a,b \in \mathbb T$ and for any vectors $\mathbf v = (v_1,
\ldots ,v_n), \mathbf w= (w_1, \ldots , w_n) \in \mathbb T^n$, we define
$$a \odot \mathbf v = (a + v_1, a + v_2, \ldots ,a + v_n),$$
$$a \odot \mathbf v \oplus b \odot \mathbf w = (\min(a+v_1,b+w_1), \ldots, \min(a+v_n,b+w_n)).$$
Let $V=\{\mathbf v_1,\dots,\mathbf v_r\}\subseteq \R^n$ be a finite set of points. The \emph{tropical convex hull} of $V$ is given by the set of all tropical linear combinations 
\[
	\tconv(V) = \{ a_1 \odot \mathbf v_1 \oplus \dots \oplus a_r \odot \mathbf v_r |~
	a_1,\dots,a_r \in \R \}.
\]
A tropically convex set in $\R^n$ is closed under tropical scalar multiplication. As a consequence, we can identify a tropically convex set $P$ contained in $\R^n$ with its image in the \emph{tropical projective torus} $ \TP^{n-1}=\R^n/(\R\odot(1,\dots,1))$.
A \emph{tropical polytope} is the tropical convex hull of a finite set $V$ in $\TP^{n-1}$. A \emph{tropical lattice polytope} is a tropical polytope whose spanning points are all contained in $\Z^n$.
Let $P~=~\tconv(V)\subseteq \TP^{n-1}$ be a tropical polytope. The \emph{(tropical)} \emph{type} of a point $\mathbf x$ in $\TP^{n-1}$ with respect to $V$ is the collection of sets $S = (S_1,\dots, S_{ n })$, where an index $i$ is contained in $S_j$ if
\[(\bv_{i})_j - x_j = \min((\bv_i)_1-x_1,\dots, (\bv_{i})_{ n }-x_{ n }).\]

Geometrically, we can view the type of $\mathbf x$ as follows: 
a \emph{max-tropical hyperplane} $H_{\mathbf a} \subseteq \TP ^{n-1}$ with \emph{apex} at $\mathbf a\in \TP^{n-1}$
is the set of points $\by \in \TP^{n-1}$ such that the maximum of $\{a_i + y_i : i \in [n] \}$ is attained at least twice. The max-tropical hyperplane $H_{\mathbf 0}$ induces a complete polyhedral fan $\mathcal F_{\mathbf 0}$ in$~\TP^{n-1}$. For each $i\in [r]$, let $H_i$ be a max-tropical hyperplane with apex $\bv_i$. Each of these hyperplanes determines a translate $\mathcal F_{H_i}$ of $\mathcal F_{\mathbf 0}$. Two points $\mathbf x,\mathbf y \in \TP^{n-1}$ lie in the same face of  $\mathcal F_{H_i}$ if and only if $\mathbf v_i - \mathbf x$ and $\mathbf v_i - \mathbf y$ achieve their minima in the same set of coordinates.
For a point $\bx\in \TP^{n-1}$ with type $S=(S_1,\dots,S_n)$, the set $S_j$ records for which hyperplanes $H_i$ the point $\bx$ lies in a face of $\mathcal F_{H_i}$ such that $\bv_i -\bx$ is minimal in coordinate $j$. \Cref{fig:types-regions} shows when $i$ is contained in $S_j$ based on the position of $\bx$ in $\TP^{2}$.

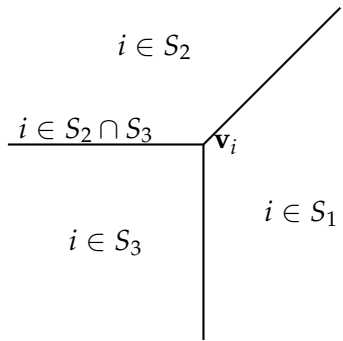
\begin{figure}[t]
    \centering

    \begin{tikzpicture}[scale=1.3]
	\draw[thick] (-2, 0) -- (0,0) -- (1.4,1.4);
	\draw[thick] (0,0) -- (0,-2);
	\node[anchor=west] at (0,0) {$\bv_i$};
	\node at (-1,-1) {$i \in S_3$};
	\node at (1,-.7) {$i \in S_1$};
	\node at (-.5,1) {$i \in S_2$};
	\node at (-1.2,.15) {$i\in S_2\cap S_3$};
\end{tikzpicture}\hspace{1cm}
    \caption{
    The max-tropical hyperplane $H_i \subseteq \TP^2$
    in the chart where the third coordinate is $0$, with faces labeled for type identification.
    }
    \label{fig:types-regions}
\end{figure}

The tropical polytope $ P$ consists of all points $\mathbf x$ whose type $S = (S_1,\dots, S_{n})$ has all $S_i$ nonempty. These are precisely the bounded regions of the subdivision of $\TP^{n-1}$ induced by the max-tropical hyperplanes $H_1,\dots H_r$ , as illustrated in \Cref{fig:trop-polytope}. Each collection of points with the same type is called a \emph{cell}. Each cell with all $S_i$ nonempty is a \emph{polytrope}: a tropical polytope that is classically convex \cite{JK}. In this way all tropical polytopes have a decomposition into polytropes. 
A tropical polytope $P$ has a unique minimal set of points $V$ such that $P = \tconv(V)$ \cite[Prop. 21]{DS}.
If $P$ is itself a polytrope, then
there is a unique maximal cell whose type with respect to $V$ is said to be the \emph{type} of the polytrope$~P$, a labeled refinement of the unlabeled combinatorial type of $P$ as a polytope.

\begin{figure}[h!]
\begin{tikzpicture}[scale=1.15]
\draw[fill=black!20, very thick] (0,-1) -- (1,-1) -- (2,0) -- (4,0) -- (5,1) -- (5,4) -- (3,4) -- (0,1) -- (0,-1);
\draw[fill=black] (4,0) circle (0.5mm);
\draw[fill=black] (5,4) circle (0.5mm);
\draw[fill=black] (1,-1) circle (0.5mm);
\draw[fill=black] (0,1) circle (0.5mm);
\node[anchor = north] at (4,0) {$(4,0,0)$};
\node[anchor = east] at (0,1) {$(0,1, 0)$};
\node[anchor = north] at (1,-1) {$(2,2,0)$};
\node[anchor = west] at (5,4) {$(5,4,0)$};
\end{tikzpicture}
\begin{tikzpicture}[scale = 1.15]
\draw[fill=black!20, very thick] (0,-1) -- (1,-1) -- (2,0) -- (4,0) -- (5,1) -- (5,4) -- (3,4) -- (0,1) -- (0,-1);
\draw[thick, dashed] (-.5, 0) -- (4,0) -- (5.4,1.4) -- (4,0) -- (4,-1.5);
\draw[thick, dashed] (-.5, 4) -- (5,4) -- (5.4,4.4) -- (5,4) -- (5,-1.5);
\draw[thick, dashed] (-.5,-1) -- (1,-1) -- (5.4,3.4) -- (1,-1) -- (1,-1.5);
\draw[thick, dashed] (-.5, 1) -- (0,1) -- (2.4,3.4) -- (0,1) -- (0,-1.5);
\draw[fill=black] (4,0) circle (0.5mm);
\draw[fill=black] (5,4) circle (0.5mm);
\draw[fill=black] (1,-1) circle (0.5mm);
\draw[fill=black] (0,1) circle (0.5mm);
\node[anchor = north west] at (4,0) {$\mathbf{4}$};
\node[anchor = south east] at (0,1) {$\mathbf{1}$};
\node[anchor = north west] at (1,-1) {$\mathbf{2}$};
\node[anchor = west] at (5,4) {$\mathbf{3}$};
\node at (2,1.5) {\scriptsize{$\{1\}, \{2,4\}, \{3\}$}};
\node at (3.5,.5) {\scriptsize{$\{1,2\}, \{4\}, \{3\}$}};
\node at (.8,-.25) {\scriptsize{$\{1\},\{2\},\{3,4\} $}};
\end{tikzpicture}    
\caption{A 2-dimensional tropical polytope in $\mathbb{TP}^2$ spanned by four vertices, pictured in the chart where the last coordinate is 0, and its decomposition into three polytropes, labeled with their types.}
\label{fig:trop-polytope}
\end{figure}
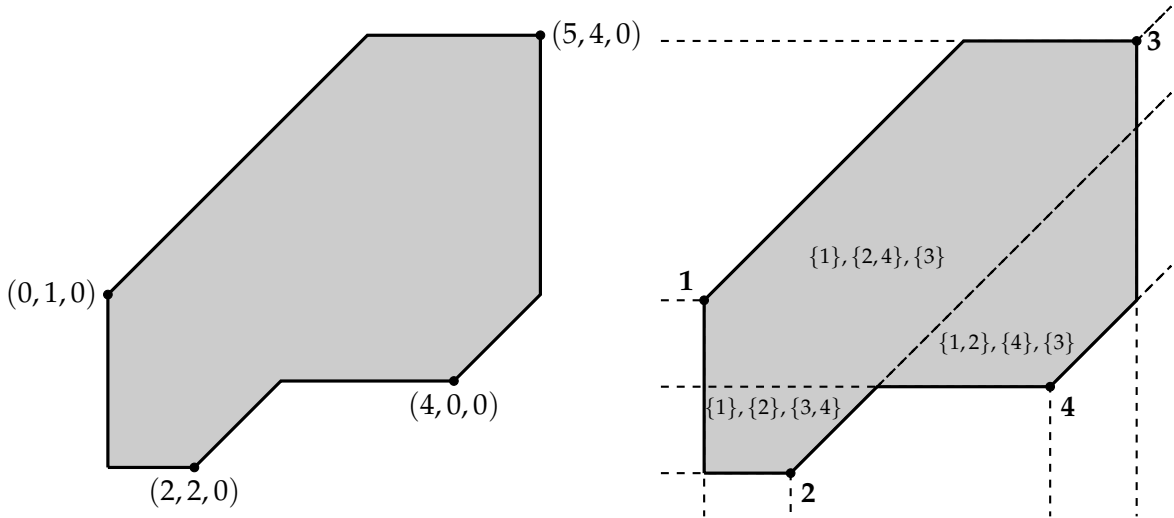

\subsection{Polytropes}
\label{subsec:background_polytropes}

We now delve deeper into a discussion of polytropes, reviewing certain results of \cite{T} and \cite{JS}.

Let $\bc$ be a vector in $\mathbb R^{n^2-n}$. We can identify $\bc$ with an $n\times n$ matrix having zeros along the diagonal. Under this identification, $\bc$ describes weights on the edges of a complete directed graph with $n$ vertices. The entry $c_{ij}$ represents the weight of the edge going from vertex $v_i$ to vertex $v_j$.

\begin{example}
Let $n=3$ and consider the vector $\bc=\left(3,2,3,4,5,6\right)\in\R^6.$ We view $\bc$ as the $3\times 3$ matrix 
    $$ \bc = \begin{pmatrix} 
        0 & 3 & 2 \\
        3 & 0 & 4 \\
        5 & 6 & 0 
    \end{pmatrix},$$ 
where each off-diagonal entry represents the weight of an edge in a complete directed graph on $3$ vertices, as shown in \Cref{fig:example_directed_graph}.

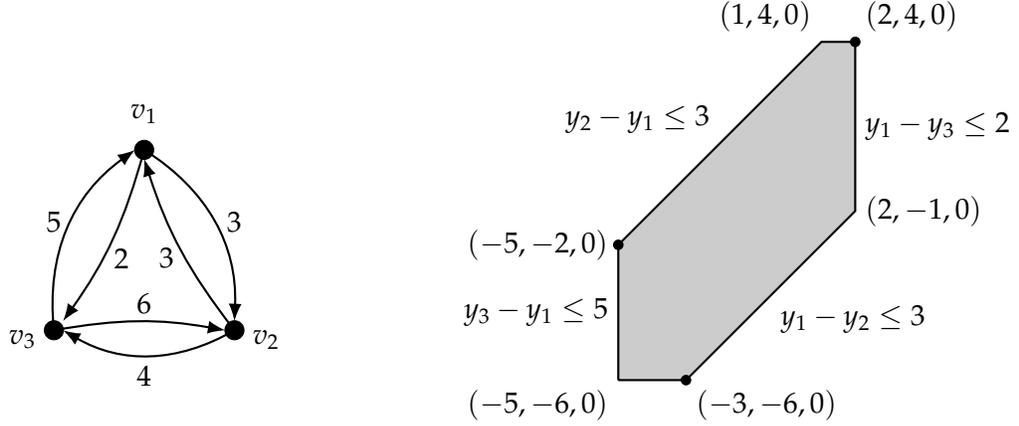
\begin{figure}
    \centering
\begin{tikzpicture}[scale=1.2]
	\draw[fill=black] (0,1) circle (3pt); 
	\draw[fill=black] (1,-1) circle (3pt); 
	\draw[fill=black] (-1,-1) circle (3pt); 
	
	\coordinate (v1) at (0,1);
	\coordinate (v2) at (1,-1);
	\coordinate (v3) at (-1,-1);

	\node at (0,1.4) {$v_1$};
	\node at (1.35, -1.1) {$v_2$};
	\node at (-1.35, -1.1) {$v_3$};
	
	\draw [thick, >=Latex, ->] (v1) to [bend left] (1,-1+0.1); 
	\draw [thick, -Latex] (v1) to [bend left=10] (-1+0.1,-1+0.1); 
	\draw [thick, -Latex] (v2) to [bend left=10] (0,1-0.1); 
	\draw [thick, -Latex] (v2) to [bend left] (-1+0.1,-1); 
	\draw [thick, -Latex] (v3) to [bend left] (0-0.1,1); 
	\draw [thick, -Latex] (v3) to [bend left=10] (1-0.1,-1); 
	
	\node at (1,.2) {$3$}; 
	\node at (-.25,-.2) {$2$}; 
	\node at (.25,-.2) {$3$}; 
	\node at (0,-1.5) {$4$}; 
	\node at (-1,.2) {$5$}; 
	\node at (0,-.7) {$6$}; 
	
	\node at (-3.5,-2) {};

\end{tikzpicture}     \hspace{2cm}
\begin{tikzpicture}[scale=.45]
    \draw[fill=black!20,thick] (-5,-6)--(-3,-6)--(2,-1)--(2,4)--(1,4)--(-5,-2)--(-5,-6);
    \node[anchor=north east] at (-5,-6) {$(-5,-6,0)$};
    \node[anchor= north west] at (-3,-6) {$(-3,-6,0)$};
    \node[anchor=west] at (2,-1) {$(2,-1,0)$};
    \node[anchor=south west] at (2,4) {$(2,4,0)$};
    \node[anchor=south east] at (1,4) {$(1,4,0)$};
    \node[anchor=east] at (-5,-2) {$(-5,-2,0)$};
    \draw[fill=black] (-3,-6) circle (4pt);
    \draw[fill=black] (2,4) circle (4pt);
    \draw[fill=black] (-5,-2) circle (4pt);
    
    \node[anchor=west] at (2,1.5) {$y_1 - y_3 \leq 2$};
    \node[anchor=north west] at (-.5,-3.5) {$y_1-y_2\leq 3$};
    \node[anchor=east] at (-5,-4) {$y_3 - y_1 \leq 5$};
    \node[anchor=south east] at (-2,1) {$y_2 - y_1 \leq 3 $};
\end{tikzpicture}

     \caption{The complete directed graph and polytrope $Q$ corresponding to the Kleene star $\bc=(3,2,3,4,5,6)$. The polytrope $Q$ is pictured in the chart where the last coordinate is zero.}
    \label{fig:example_directed_graph}
\end{figure}
\end{example}

We define $\mathcal R_n\subseteq \mathbb R^{n^2-n}$ to be the set of all vectors $\bc$ with no negative cycles in the corresponding weighted graph. The \emph{Kleene star} $\bc^*\in\R^{n\times n}$ of $\bc$ is the matrix such that $\bc_{ij}^*$ is the weight of the lowest-weight path from $i$ to $j$. It can be computed as the $(n-1)$th tropical power $\bc^{\odot (n-1)}$. 
Since $\bc$ has no negative cycles, $\bc^*$ is zero along the diagonal, and we can again identify $\bc^*$ with a vector in $\R^{n^2-n}$. 
The \emph{polytrope region} $\Pol_n \subseteq \mathcal R_n\subseteq\RR^{n^2-n}$ is the closed cone given by
\[\Pol_n = \{\bc\in \mathcal R_n: \bc=\bc^*\}.\]
Points in the polytrope region correspond to weighted graphs whose edges satisfy the triangle inequality. As the name suggests, the polytrope region parametrizes the set of all polytropes:

\begin{prop}[{\cite[Th. 1]{dlP}, \cite[Prop. 13]{T}}]\label{prop:polytrope_vertex_hyperplane_description}
Let $P\subseteq \TP^{n-1}$ be a non-empty set. The following statements are equivalent:
\begin{enumerate}
    \item $P$ is a polytrope.
    \item There is a matrix $\mathbf c\in\Pol_n$ such that $P=\tconv(\mathbf c)$, where the columns of the matrix $\mathbf c$ are taken as a set of $n$ points in $\TP^{n-1}$.
    \item There is a matrix $\mathbf c\in\Pol_n$ such that \[P=\{y\in\R^{n}\mid y_i - y_j \leq c_{ij}, y_{n}=0\}.\]
\end{enumerate}
Furthermore, the $\mathbf c$'s in the last two statements are equal, and are uniquely determined by $P$.
\end{prop}

Note in particular that polytropes in $\TP^{n-1}$ are tropical simplices, i.e. the tropical convex hull of exactly $n$ points. A polytrope of dimension $n-1$ is \emph{maximal} if it has $\binom{2n-2}{n-1}$ vertices as an ordinary polytope. To see why this is indeed the maximal number of classical vertices, we note that a polytrope is dual to a regular subdivision of the product of simplices $\Delta_{n-1}\times\Delta_{n-1}$. The normalized volume of this polytope is $\binom{2n-2}{n-1}$, bounding the number of maximal cells in the regular subdivision and hence the number of vertices of the polytrope. This bound is attained in every dimension \cite[Proposition 19]{DS}.

Let $R$ be the polynomial ring $R=\R[x_{ij} \mid (i,j)\in [n]^2,\ i\neq j]$. Given a vector $\bv=(v_{12},\dots, v_{n(n-1)})$ contained in $\mathbb{N}^{n^2 - n}$, we write $\bx^\bv$ for the monomial $\prod x_{ij}^{v_{ij}}$. A vector $\bc\in\mathbb R^{n^2-n}$ determines a partial ordering $>_\bc$ on the monomials of $R$, where monomials are compared using the dot product of their exponent vector with $\bc$, i.e. $\bx^\bu {>_\bc} \bx^\bv$ if $\bu\cdot \bc > \bv\cdot \bc$. Given a polynomial $f=\sum_\bv \alpha_\bv \bx^\bv$, some of its monomial terms will be maximal with respect to this partial ordering. We define the initial term $in_\bc(f)$ to be the sum of all such maximal terms of $f$. The initial ideal $in_\bc(I)$ of an ideal $I\subseteq R$ is generated by all initial terms $in_\bc(f)$ for $f\in I$. In general, a generating set $\{f_i\}$ for the ideal $I$ need not satisfy $\langle in_{\bc}(f_i)\rangle = in_\bc (I)$. If $\langle in_{\bc}(f_i)\rangle = in_\bc (I)$ does hold, we call the generating set $\{f_i\}$ a \emph{Gr\"obner basis} of $I$ with respect to the weight vector $\bc$.
For a more detailed introduction to Gr\"obner bases, see \cite{coxlittleoshea}.

We consider the toric ideal
\[ I 
= \langle x_{ij}x_{ji}-1, x_{ij}x_{jk}-x_{ik}\mid i, j, k \in [n] \text{ pairwise distinct}\rangle
,\]
which appears in \cite{T} as the toric ideal associated with the all-pairs shortest path program. 
Let $\bc \in \mathbb R^{n^2-n}$. The \emph{Gr\"obner cone} $\mathcal C_\bc(I)$ is given by
\[\mathcal C_\bc(I) = \{\bc'\in \
\mathbb R^{n^2-n}: in_{\bc'}(I) = in_\bc(I)\} .\]
This is a closed, convex polyhedral cone. The collection of all such cones is a polyhedral fan, the \emph{Gr\"obner fan} $\mathcal{GF}_n$ of the ideal $I$. Let $\mathcal{GF}_n|_{\mathcal Pol_n}$ be the restriction of the Gr\"obner fan of $I$ to the polytrope region $\mathcal Pol_n$. This polyhedral fan captures the tropical types of polytropes:

\begin{theorem}[{\cite[Th. 17 - 18]{T}}]\label{th:polytropes:comb:types}
Cones of $\mathcal{GF}_n|_{\mathcal Pol_n}$ are in bijection with tropical types of polytropes in $\TP^{n-1}$. Open cones of $\mathcal{GF}_n|_{\mathcal Pol_n}$ are in bijection with types of maximal polytropes in $\TP^{n-1}$.
\end{theorem}

Up to the action of the symmetric group $S_3$ on the labels of the vertices, in dimension $2$ there is precisely one maximal tropical type of polytrope, namely the hexagon. In dimension $3$ there are $6$ distinct maximal tropical types up to symmetry \cite{JK, JDlP}. Using \Cref{th:polytropes:comb:types}, \cite{T} showed that in dimension $4$ there are $27248$ distinct types up to the symmetric group action. In higher dimensions, this number is unknown. These tropical type counts were independently confirmed in \cite{JS} using the following identification:

\begin{prop}[{\cite[Theorem 1 and Lemma 7]{DS}}]
Let $V = \{\bv_1,\dots,\bv_r\} \subseteq \TP^{n-1}$.
There is a piecewise-linear isomorphism between the tropical polytope $\tconv(V)$ and the polyhedral complex of bounded faces of the unbounded polyhedron
\[
    \mathcal{P}_V = \{ (\by,\bz)\in \R^{r+n} / (1,\dots,1,-1,\dots,-1) \mid y_i + z_i \leq v_{ij} \text{ for all } i\in[r], j\in[n] \}.
\]
The boundary complex of $\mathcal{P}_V$ is polar to the regular subdivision of the products of simplices $\Delta_{r-1}\times\Delta_{n-1}$ defined by the weights $v_{ij}$.
\end{prop}

In particular, the bounded faces of $\mathcal P_V$ are dual to the interior faces of the regular subdivision of $\Delta_{r-1}\times\Delta_{n-1}$, i.e. the faces not completely contained in the boundary of $\Delta_{r-1}\times\Delta_{n-1}$. If $\tconv(V)$ is a polytrope, then \Cref{prop:polytrope_vertex_hyperplane_description} implies that $r=n$.
Even more, $\tconv(V)$ is a polytrope if and only if the bounded region of $\mathcal P_V$ consists of a single bounded face \cite[Th. 15]{DS}, and hence all maximal cells in the dual subdivision of $\Delta_{n-1}\times\Delta_{n-1}$ share some vertex.

By the Cayley trick \cite{HRS}, this is identical to studying mixed subdivisions of the dilated simplex $n\cdot \Delta_{n-1}$. Regular subdivisions of products of simplices can thus be related to certain regular subdivisions of the \emph{fundamental polytope} $FP_n$, a subpolytope of $n\cdot \Delta_{n-1}$ introduced by Vershik \cite{V} and further studied by Delucchi and Hoessly \cite{DH}. Their central subdivisions were studied in \cite{ardila11_rootpolytopesgrowth} as subdivisions of the boundary of the root polytope of type $A$.
\[
    FP_n = \conv\{\be_i - \be_j \mid i\ne j\in [n] \}.
\]
\begin{figure}[h!]
    \centering
    \includegraphics[height=3in]{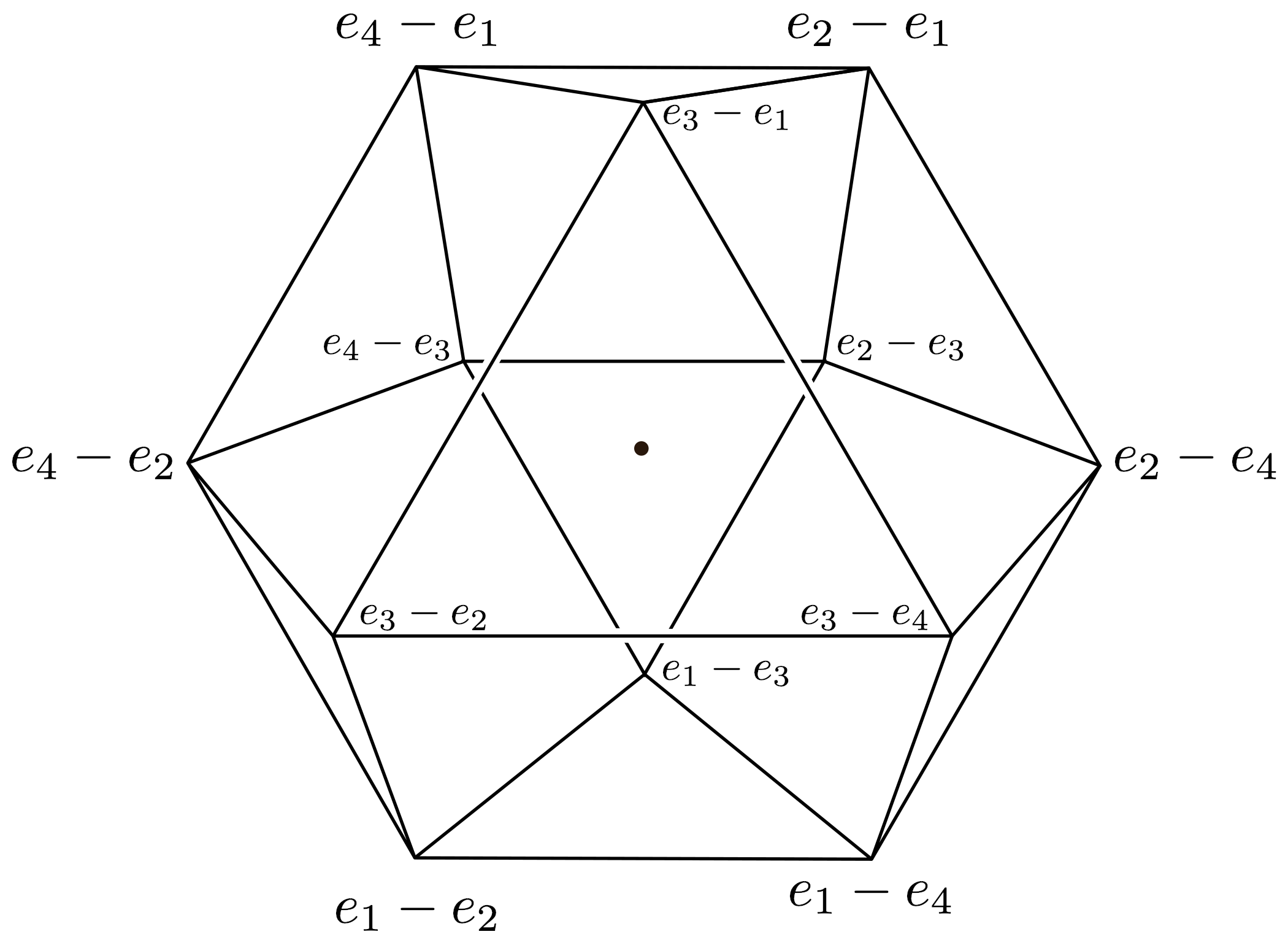}
    \caption{The fundamental polytope $FP_4$ with unique interior lattice point $\mathbf 0$, whose regular central subdivisions correspond to tropical types of 3-dimensional polytropes.}
    \label{fig:3d_fundamental_polytope}
\end{figure}
The fundamental polytope $FP_4$ is pictured in \Cref{fig:3d_fundamental_polytope}. A \emph{regular central subdivision} of $FP_n$ is a regular subdivision in which the unique relative interior lattice point $\mathbf 0$ of $FP_n$ is lifted to height $0$ and is a vertex of each maximal cell. The number of tropical types can be enumerated using the following theorem:
\begin{theorem}[{\cite[Th. 22]{JS}}] The tropical types of full-dimensional polytropes in $\TP^{n-1}$ are in bijection with the regular central subdivisions of $FP_n$. 
\end{theorem}
We connect our computational results to regular central subdivisions of the fundamental polytope in \Cref{subsec:3D-vols,subsec:4D-vols}.

\subsection{Toric geometry}
\label{subsec:toric_geometry}
In order to compute multivariate volume polynomials of polytropes we use methods from toric geometry. We now give a brief summary of the toric geometry needed in our computation. For further details, the reader may consult \cite[Ch. 12.4, 13.4]{cox2011toric} or \cite[Ch. 5.3]{fulton}.

Let $P$ be a lattice polytope with $m$ facets, so that $P$ is given by
 \[ P = \{ \bx\in\R^n \mid \langle \bx, \bu_i \rangle + c_i \geq 0 \text{ for } i\in[m]\}, \]
where $\bu_i$ is the primitive facet normal of the facet $F_i$. 
We denote by $\Sigma$ the normal fan of $P$, and by $X$ the toric variety defined by the fan $\Sigma$.
We assume that $X$ is smooth, so $\Sigma$ is a simplicial fan and $P$ a simple polytope.

Let $d=\dim(X)$.
A torus-invariant prime divisor $D_i$ of $X$ is a subvariety of $X$ of codimension $1$, which is in bijection with a ray of $\Sigma$ and hence also with a facet $F_i$ of $P$. Given the polytope $P$, we can define the divisor $D_P$ as the linear combination $D_P = \sum_{i=1}^m  c_i D_i$. At the same time, as $D_i$ is an irreducible  subvariety of $X$ of codimension $1$, it gives rise to a cohomology class $[D_i] \in H^2(X,\Q)$ and $[D_P] = [\sum_{i=1}^m c_i D_i] = \sum_{i=1}^m c_i [D_i] \in H^2(X,\Q)$. 

Given irreducible subvarieties $V,W \subseteq X$ with $\dim(V) = k_1,  \dim(W) = k_2$, we can consider the cup product $[V]\smile [W] \in H^{2(k_1 + k_2)}(X,\Q)$. If $k_1 + k_2 = d$, then by Poincare-duality $[V]\smile [W] \in H^{2d}(X,\Q) \cong H_0(X,\Q) \cong \Q$ and thus we can define the \emph{integral (or intersection product)} 
$\int_X ([V] \smile [W]) \in \Q.$ We use $[V]^2$ as shorthand notation for $([V] \smile [V])$.

\begin{theorem}[{\cite[Theorem 13.4.1]{cox2011toric}}]\label{th:cohomology-volume}
	The normalized volume of $P$ is given by
	$$\Vol(P) = \int_X \left[\sum_{i=1}^m c_i D_i\right]^d.$$
\end{theorem} 
In order to be able to compute these polynomials systematically, we make use of an identification of the cohomology ring as a polynomial ring.
Let $K$ be a field of characteristic 0 and $S$ be a simplicial complex on $m$ vertices. The \emph{Stanley-Reisner ideal} $M$ in the polynomial ring $R = K[x_1,\dots,x_m]$ is the ideal generated by the (inclusion-minimal) non-faces of $S$, i.e.
\[
M = \langle x_{i_1}\cdots x_{i_k} \mid \{i_1,\dots,i_k\} \text{ is not a face of S } \rangle.
\]
Let $\mathcal B$ be a basis of $\Z^n$. 
Since $\Sigma$ is simplicial, we can consider the \emph{Stanley-Reisner ideal} $M$ of $\Sigma$, i.e. the Stanley-Reisner ideal of the boundary complex $\partial P^\circ $ of the polar of $P$.
The cohomology ring $H^*(X,\mathbb{Q})$ is isomorphic to the quotient ring $R/(L+M)$, where $L$ is the ideal 
\[ L = \left\langle \sum_{i=1}^m \langle \mathbf b,\bu_{i} \rangle x_{i} \, \middle |\, \mathbf b \in \mathcal B \right\rangle.  \]

The variable $x_i$ in $R/(L+M)$ corresponds to  $[D_i]$, the cohomology class of a torus-invariant prime-divisor, and hence to a facet of $P$. 
Therefore, the expression in \Cref{th:cohomology-volume} translates to a polynomial
\[
 \left(\sum_{i=1}^m c_i x_i\right)^{d} \in R/(L+M).
\]
 The top cohomology group is a one-dimensional vector space. A canonical choice of a basis vector in $R/(L+M)$ is any square-free monomial $\mathbf{x}^{\bm{ \alpha}}$ which indexes a vertex of $P$. The expression $(\sum_{i=1}^m c_i x_i)^{\dim(X)}$ has a representation $\delta \cdot \mathbf{x}^{\bm \alpha}$ in $R/(L+M)$. 
 The volume of $P$ will be given by the coefficient $\delta$, up to a correcting factor that solely depends on the choice of the basis $\mathbf{x}^{\bm \alpha}$ \cite[Algorithm 1]{DLS}.

Replacing the values $c_1,\dots,c_m$ defining the facets of $P$ by indeterminates $a_1,\dots,a_m$ we obtain a polynomial which gives the volume $\Vol(P)$ of the polytope $P$ when evaluated at $c_1,\dots,c_m$. 
We hence refer to such a polynomial as a \emph{volume polynomial} of $P$. This polynomial depends only on the normal fan of $P$, and so polytopes with the same normal fan determine the same volume polynomial.
\Cref{alg:cohomology_class} describes how to compute the integral of a cohomology class of $X$.

\section{Computing multivariate polynomials}
\label{sec:methods}

In this section we introduce our multivariate polynomials of interest and describe methods for computing these functions for polytropes, motivated by the methods in \cite{T} and \cite{DLS}.

\subsection{Computing multivariate volume polynomials}\label{subsec:vol_pols_computation}

We seek to compute a \emph{multivariate volume polynomial} for each tropical type of polytropes as discussed in \Cref{subsec:toric_geometry}: that is, a polynomial in variables $a_{ij}$ for each tropical type which evaluates to the volume of a polytrope $P(\bc)$ of the appropriate type when given the respective Kleene star $\bc$. For each tropical type, our computation of such a polynomial will depend on a fixed Kleene star $\bc$ of the appropriate type.

Consider the ``indeterminate polytrope''
\[P(\ba)=\{\by\in\R^{n}\mid y_i - y_j \leq a_{ij}, y_{n}=0\},\]
defined by indeterminates $a_{ij}$. By \Cref{prop:polytrope_vertex_hyperplane_description}, $a_{ij}$ is the weight of the shortest path in a weighted complete digraph. 
As $P(\ba)$ is contained in the linear space given by $y_n = 0$, we can project onto the first $n-1$ coordinates, which yields
\[ P(\ba)=\{\by\in\R^{n-1}\mid B\by \leq \mathbf{a}\} \]
for a suitable matrix $B\in\Z^{(n^2-n)\times (n-1)}$
with rows $B_{ij}$ indexed by $ij \in [n]^2, i\neq j$. This is a representation of $P(\ba)$ given by $n^2-n$ inequalities $B_{ij} \by \leq a_{ij}$ and $\by = (y_1, \dots, y_{n-1})$, as in \Cref{ex:2d-B-description} below.
Note that for each $j\in[n-1]$ there is an inequality $-y_j \leq a_{nj}$.
We introduce a nonnegative slack-variable $y_{ij}$ for each inequality and replace the inequality by the equation $y_i - y_j + y_{ij} = a_{ij}$.
This gives a representation as $(B \mid  Id_n) \by = \ba$ with $\by = (y_1,\dots, y_{n-1}, y_{12},\dots, y_{n-1,n})$.

In particular, we have the equation $-y_j + y_{nj} = a_{nj}$. We can thus substitute the variable $y_j, j\in~[n~-~1]$ by $y_{nj}-a_{nj}$, which leaves us with a system of equations of the form 
\begin{align*}
y_{ij}+y_{ni}-y_{nj} &= a_{ij}+a_{ni}-a_{nj}  \quad  (1) \\ 
y_{jn} + y_{nj} &= a_{jn} + a_{nj}.  \qquad \quad (2) 
\end{align*}
Adding these equations gives $y_{ij} + y_{ni} + y_{jn} = a_{ij} +a_{ni} +a_{jn}$  (1'). The set of solutions to the system with equations (1) and (2) is equal to the set of solutions to the system with (1') and (2), yielding a matrix $A$ such that 
\[ P(\ba) = \{\by\in\R^{n^2-n}_{\geq 0} \mid A\by = A\mathbf a \}\]
and $\ker(A)\cap \R^{n^2-n}_{\geq 0} = 0 $, thus fulfilling the general assumptions in \cite{DLS}. The expressions $a_{ij}+a_{ni}+a_{jn}$ and $a_{in} + a_{ni}$ have a nice interpretation in terms shortest paths of the complete digraph: these are the weights of the shortest cycle passing through $i$ and $n$ and the shortest directed cycle passing through $i,j$ and $n$ respectively.
A matrix is called \emph{totally unimodular} if every minor equals $-1, 0,$ or $1$. A full-dimensional lattice polytope in $\R^n$ is called \emph{unimodular} if each of its vertex cones is generated by a basis of $\Z ^n$. 
This condition is sometimes referred to as \emph{smooth} or \emph{Delzant}. It is well-known that our constraint matrix $A$ is totally unimodular \cite[Section 2.3.3]{T}, and that maximal polytropes are unimodular and simple \cite[Section 7.3]{handbook}.

\begin{example}
\label{ex:2d-B-description}
Any $2$-dimensional polytrope $P(\ba)$ has an $H$-description as 

\[ P(\ba) = \left\{  
    \begin{pmatrix} y_1 \\ y_2 \end{pmatrix} \in \RR^2 \, \middle|\, \begin{aligned} \, 
                y_1-y_2 \leq a_{12}, y_2-y_1 \leq a_{21}  \\
                y_1 \leq a_{13}, y_2\leq a_{23} \\
                y_1 \geq - a_{31}, y_2\geq -a_{32}  \end{aligned}
 \right\},
\]
when $\mathbf a$ is contained in the polytrope region $\mathcal Pol_3$.
We want to compute the constraint matrix $A$ by turning the above description of a polytrope into one involving only equalities. We begin by translating the above to a matrix description of $P(\ba)$:

\[\begin{pmatrix}
1 & -1  \\
1 & 0  \\
-1 & 1   \\
0 & 1  \\
-1 & 0   \\
0 & -1  \\
\end{pmatrix}
\begin{pmatrix}
y_1 \\ y_2 
\end{pmatrix}
\leq 
\begin{pmatrix}
a_{12} \\ a_{13}  \\ a_{21} \\ a_{23}  \\ a_{31} \\ a_{32} 
\end{pmatrix}
\]
Introducing slack variables $y_{ij}$, we get the representation

\[\begin{pmatrix}
 1 & -1 & 1 & 0 & 0 & 0 & 0 & 0 \\
 1 &  0 & 0 & 1 & 0 & 0 & 0 & 0 \\
-1 &  1 & 0 & 0 & 1 & 0 & 0 & 0 \\
 0 &  1 & 0 & 0 & 0 & 1 & 0 & 0 \\
-1 &  0 & 0 & 0 & 0 & 0 & 1 & 0 \\
 0 & -1 & 0 & 0 & 0 & 0 & 0 & 1 \\
\end{pmatrix}
\begin{pmatrix}
y_1\\y_2\\y_{12}\\y_{13}\\y_{21}\\y_{23}\\y_{31}\\y_{32}\end{pmatrix}
=
\begin{pmatrix}
a_{12} \\ a_{13} \\ a_{21} \\ a_{23} \\ a_{31} \\ a_{32}
\end{pmatrix}\]
Substituting $y_1 = y_{31} - a_{31}, y_2 = y_{32} - a_{32}$ and deleting zero-columns and zero-rows gives us
\[\begin{pmatrix}
1 & 0 & 0 & 0 &  1 & -1 \\
0 & 1 & 0 & 0 &  1 &  0 \\
0 & 0 & 1 & 0 & -1 &  1 \\
0 & 0 & 0 & 1 &  0 &  1 \\
\end{pmatrix}
\begin{pmatrix}
y_{12}\\y_{13}\\y_{21}\\y_{23}\\y_{31}\\y_{32}\end{pmatrix}
=
\begin{pmatrix}
a_{12}+a_{31}-a_{32} \\ a_{13}+a_{31} \\ a_{21}-a_{31}+a_{32} \\ a_{23}+a_{32}
\end{pmatrix}  \]
This is equivalent to
\[A\by = \begin{pmatrix}
1 & 0 & 0 & 1 &  1 &  0 \\
0 & 1 & 0 & 0 &  1 &  0 \\
0 & 1 & 1 & 0 &  0 &  1 \\
0 & 0 & 0 & 1 &  0 &  1 \\
\end{pmatrix}
\begin{pmatrix}
y_{12}\\y_{13}\\y_{21}\\y_{23}\\y_{31}\\y_{32}\end{pmatrix}
=
\begin{pmatrix}
a_{12}+a_{23}+a_{31} \\ a_{13}+a_{31} \\ a_{21}+a_{13}+a_{32} \\ a_{23}+a_{32}
\end{pmatrix} = A\mathbf{a}.\]
This gives us the desired representation $P(\ba) = \{\by\in\R^{n^2-n}_{\geq 0} \mid A\by = A\mathbf a \}$.
\end{example}

Let $K=\Q(a_{ij} \mid (i,j)\in [n]^2,\ i\neq j)$,  for $a_{ij}$ indeterminate variables, and consider the polynomial ring ${R=K[x_{ij} \mid (i,j)\in [n]^2,\ i\neq j]}$.
Following \cite{DLS}, we consider the toric ideal seen previously in \Cref{subsec:background_polytropes}:
\[ I = \langle \mathbf x^{\mathbf r} - 1 \mid \mathbf r \text{ is a row of } A \rangle
= \langle x_{in}x_{ni}-1, x_{ij}x_{jn}x_{ni}-1\rangle
= \langle x_{ij}x_{ji}-1, x_{ij}x_{jk}-x_{ik}\rangle
,\] 
where $(i,j,k)\in[n]^3$ are pairwise distinct.

Fix a tropical type and Kleene star $\bc$ corresponding to a polytrope $P(\bc)$ of that type. We write $M=in_\mathbf{c}(I)$ for the initial ideal of $I$ with respect to the weight vector $\mathbf c$. 
\begin{prop}[{\cite[Prop. 2.3]{DLS}}]
	The initial ideal $M$ is the Stanley-Reisner ideal of the normal fan $\Sigma$ of the simple polytope $P(\bc)$.
\end{prop}
In order to compute the volume polynomial, we need to know the minimal primes of $M$. 
\begin{prop}[{\cite[Lemmas 5 and 6]{block2005tropical}}, \cite{DLS}]\label{prop:min_primes}
	The facets of $P(\bc)$ are in bijection with variables $x_{ij}$. The vertices of $P(\bc)$ are in bijection with minimal primes of $M$. 
\end{prop}

In the above bijection, the facet $F_{ij}$ given by the inequality $y_i - y_j = c_{ij}$ is identified with the variable $x_{ij}$.
A vertex $v$ of the polytrope can be identified with the minimal prime $\langle x_{ij} \mid ij \not\in ~ \mathcal I_v \rangle$, where $\mathcal I_v = \{ ij \mid F_{ij} \text{ contains } v \}$. Thus, a minimal prime is generated by variables which correspond to facets that do not contain a given vertex $v$.

Let $X$ be the smooth toric variety defined by the normal fan $\Sigma$ of the unimodular polytope $P(\bc)$. 
Let $\bu_{ij}$ denote the primitive ray generators of the normal fan of $P(\bc)$, i.e.  $\bu_{ij} = \be_j - \be_i$ for $i\ne j \in [n-1]\times[n-1]$ and $\bu_{ni} = \be_i,\ \bu_{in}= -\be_i$ for $i\in[n-1]$. Let further $\mathcal B$ be a basis for $\Z^n$.
Recall that the cohomology ring $H^*(X,\mathbb{Q})$ is isomorphic to the quotient ring $R/(L+M)$, where $M$ is the initial ideal from above and $L$ is the ideal
\[ L = \left\langle \sum_{\substack{ij \in [n]\times[n] \\ i\neq j}} \langle \mathbf b,\bu_{ij} \rangle x_{ij} \, \middle |\, \mathbf b\in\mathcal B \right\rangle.  \]
 Choosing $\mathcal B$ to be the standard basis for $\Z^n$, for a given vector $\mathbf b=\be_k$ we get 
 \[
      \sum_{\substack{ij \in [n]\times[n] \\ i\neq j}} \langle \mathbf e_k,\bu_{ij} \rangle x_{ij} 
     =  \sum_{j\in[n]} x_{kj} - \sum_{j\in[n]} x_{jk}
\]
 and so the ideal is equal to 
 \[ L = \left\langle \sum_{j\in[n]} x_{kj} - \sum_{j\in[n]} x_{jk} \, \middle | \, k\in[n] \right\rangle.  \]
 Considering the complete directed graph $K_n$ on $n$ vertices, this ideal can be viewed as generated by the cuts of $K_n$ that isolate a single vertex.
 
Let $D$ be the divisor on $X$ corresponding to the polytrope $P(\ba)$ given by the indeterminates $a_{ij}$, i.e. $P(\ba) = \{\by \in \R^{n} \mid y_i - y_j \leq a_{ij} , y_n=0\}$.
We can write $D$ as
	 $$D=\sum_{\substack{ij \in [n]\times[n] \\ i\neq j}} a_{ij}D_{ij},$$ 
where $D_{ij}$ is the prime divisor corresponding to the ray of $\Sigma$ spanned by $\bu_{ij}$.
 Let $$q=\sum_{\substack{ij \in [n]\times[n] \\ i\neq j}} a_{ij}x_{ij}$$ be the polynomial in $R$ representing the divisor $D$.
 
 In the following we present an algorithm to compute the integral of a top cohomology class of $X$. As the dimension of a polytrope $P(\bc)$ defined by a Kleene star $\bc\in \mathcal{GF}|_{\mathcal{P}ol_n}$ is $n-1$, we can compute the volume polynomial restricted to an open maximal cone of $\mathcal{GF}|_{\mathcal{P}ol_n}$ by
 	\[ \Vol_{P(\ba)}(\mathbf a) = \int_{X} [D]^{n-1} . \]
Note that the integral $\int_{X} [D]^{n-1}$ is a constant in $R$ and thus a polynomial with variables $a_{ij}$, as discussed in \Cref{subsec:toric_geometry}.  If the input of \Cref{alg:cohomology_class} is given by the polynomial $p=q^{n-1}$, the output is a multivariate volume polynomial. 
\\
\begin{algorithm}[Computing the integral of a cohomology class of $X$ {\cite[Alg. 1]{DLS}}] \label{alg:cohomology_class}
\begin{algorithmic}[1]
\item[]
    \REQUIRE A polynomial $p(x)$ with coefficients in a field $k\supset \Q$.
    \ENSURE The integral $\int_{X} p$ of the corresponding cohomology class on $X$.
    \STATE Compute a Gr\"obner basis $\mathcal G$ for the ideal $M+L$.
    \STATE Find a minimal prime $\langle x_j \, \mid  \,x_j \not\in \mathcal I_v \rangle $  of $M$, and compute the normal form of $\prod_{i\in \mathcal I_v} x_i$ modulo the Gr\"obner basis $\mathcal G$. It looks like $\gamma \cdot \mathbf{x}^{\bm \alpha}$, where $\gamma$ is a non-zero element of $k$ and $\mathbf{x}^{\bm \alpha}$ is the unique standard monomial of degree $n-1$.
    \STATE Compute the normal form of $p$ modulo $\mathcal G$ and let $\delta\in k$ be the coefficient of $\mathbf{x}^{\bm \alpha}$ in that normal form.
    \STATE Output the scalar $\delta/\gamma \in k$.
\end{algorithmic}
\end{algorithm}

The correctness of the algorithm follows from \cite[Section 2]{DLS}.
The discussion above implies the following theorem.

\begin{theorem}
Consider the ``indeterminate polytrope''
\[P(\ba)=\{\by\in\R^{n}\mid y_i - y_j \leq a_{ij}, y_{n}=0\},\] with a fixed normal fan.
Let
$$q=\sum_{\substack{ij \in [n]\times[n] \\ i\neq j}} a_{ij}x_{ij}$$ be the polynomial in $R$ in variables $x_{ij}$ and indeterminates $a_{ij} \in K$. Let $p = q^{n-1}$ be the input of  \Cref{alg:cohomology_class}. The  output $\frac{\delta}{\gamma} \in K$ is the multivariate volume polynomial of $P(\ba)$ in variables $a_{ij}$, i.e.
$$
    \Vol_{P(\ba)}(\ba) = \frac{\delta}{\gamma}.
$$
\end{theorem}

Recall that $\mathcal{GF}_n|_{\mathcal Pol_n}$ is the restriction of the Gr\"obner fan of $I$ to the polytrope region $\mathcal Pol_n$. By \Cref{th:polytropes:comb:types}, open cones of $\mathcal{GF}_n|_{\mathcal Pol_n}$ are in bijection with types of maximal polytropes. 
Since \Cref{alg:cohomology_class} only depends on $in_{\mathbf c}(I)$ and not the choice of $\bc$ itself, the multivariate volume polynomial is constant along an open cone of $\mathcal{GF}_n|_{\mathcal Pol_n}$. This reflects the fact that polytropes of the same tropical type have the same normal fan. Therefore, maximal polytropes of the same type have the same multivariate volume polynomial, and it suffices to compute the polynomial for only one representative $\mathbf c$ for each maximal cone.
Furthermore, the polynomials agree on the intersection of the closure of two of these cones \cite{vector_partition_functions}. Thus, given a Kleene star $\mathbf c$ corresponding to a non-maximal polytrope $P(\bc)$, we can choose any of the maximal closed cones that contain $\mathbf c$ and evaluate the corresponding multivariate volume polynomial at $\mathbf c$ to compute the volume of~$P(\bc)$.

\begin{example}\label{ex:vol_pol_2d}

We apply the above discussion to compute the multivariate volume polynomial for 2-dimensional polytropes. Note that the volume, Ehrhart- and $h^\ast$-polynomial of the hexagon can be derived by more elementary methods as, for example, counting unimodular simplices in an alcoved triangulation and Pick's formula. However, as the presentation is less clear in dimensions $3$ and $4$, we showcase the algebraic machinery on this example. The toric ideal $I$ is 
\[I = \langle x_{12}x_{23}x_{31}-1,\ x_{13}x_{31}-1,\ x_{21}x_{13}x_{32}-1,\ 
x_{23}x_{32}-1\rangle.\]
We also have $L$ as 
\[L = \langle 
    x_{12}+x_{13}-x_{21}-x_{31},\
    x_{21}+x_{23}-x_{12}-x_{32},\
    x_{31}+x_{32}-x_{13}-x_{23}
\rangle.\] 
Let $\mathbf c=(3,2,3,4,5,6)$. Then the corresponding polytrope $\Qa[\bc]$ is the hexagon displayed in \Cref{fig:example_hexagon}, with facets labeled according to \Cref{prop:min_primes}.

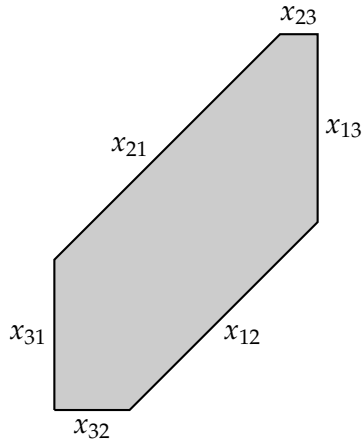
\begin{figure}[h]
    \centering
\begin{tikzpicture}[scale=.5]
    \draw[fill=black!20,thick] (-5,-6)--(-3,-6)--(2,-1)--(2,4)--(1,4)--(-5,-2)--(-5,-6);
    \node at (-4,-6.5) {$x_{32}$};
    \node at (0,-4) {$x_{12}$};
    \node at (2.7,1.5) {$x_{13}$};
    \node at (1.5,4.5) {$x_{23}$};
    \node at (-3,1) {$x_{21}$};
    \node at (-5.7,-4) {$x_{31}$};
\end{tikzpicture}

     \caption{The polytrope $Q$ corresponding to the vector $\mathbf c=(3,2,3,4,5,6)$. }
    \label{fig:example_hexagon}
\end{figure}

The initial ideal $M$ of $I$ with respect to the weight vector $\mathbf c$ is
\[ M = \langle x_{12}x_{21},\ x_{13}x_{21},\ x_{12}x_{23},\ x_{12}x_{31},\ x_{13}x_{31},\ x_{23}x_{31},\ x_{13}x_{32},\ x_{21}x_{32},\ x_{23}x_{32} \rangle . \]

A Gröbner basis for $M+L$ is given by
\begin{align*}
    \mathcal{G} = \langle
    & x_{31}-x_{12}+x_{21}-x_{13}, \ x_{13} x_{21}, \ x_{12} x_{13}+x_{13}^2, 
    \ x_{32}-x_{23}+x_{12}-x_{21}, \ x_{13} x_{23}+x_{13}^2, \\ 
  & x_{21}^2-x_{13}^2, \ x_{12} x_{21}, \ x_{12}^2-x_{13}^2, \
  x_{13}^3, \ x_{21} x_{23}+x_{13}^2, \ x_{12} x_{23}, \ x_{23}^2-x_{13}^2
    \rangle.
\end{align*}
Any vertex gives us a minimal prime. We choose the vertex $v$ incident to the facets labeled by $x_{31}$ and $x_{32}$, giving us the minimal prime
$\langle x_{ij} \mid ij \not\in \mathcal I_v \rangle  = \langle x_{12}, x_{13}, x_{21}, x_{23} \rangle$ and the monomial $\prod_{ij \in \mathcal I_v} x_{ij} = x_{31}x_{32}$. Modulo the Gröbner basis $\mathcal G$, this is $\gamma \cdot \mathbf{x}^{\bm \alpha} = (-1)x_{13}^2$, so $\gamma=-1$ and $\mathbf{x}^{\bm \alpha}=x_{13}^2$.

Let $q=\sum_{\substack{ij \in [n]\times[n] \\ i\neq j}} a_{ij}x_{ij}$. This is the polynomial in $R/(L+M)$ corresponding to the divisor described in \Cref{subsec:vol_pols_computation}. We want to compute the volume of the polytrope $\Qa[\bc]$. This can be done by applying \Cref{alg:cohomology_class} to $p = q^2$.

The polynomial $q^2$ modulo $\mathcal G$ is 
\begin{align*}
    & (a_{12}^2 - 2a_{12}a_{13} + a_{13}^2 + a_{21}^2 - 2a_{13}a_{23}  - 2a_{21}a_{23} \\
     + & a_{23}^2 - 2a_{21}a_{31}  + a_{31}^2 - 2a_{12}a_{32} - 2a_{31}a_{32} + a_{32}^2)  x_{13}^2,
\end{align*}
so the coefficient $\delta$ of $\mathbf{x}^{\bm \alpha}$ gives us the volume polynomial for the normalized volume
\begin{align*} \Vol_{\Qa}(\mathbf a) = \frac{\delta}{\gamma} =   -( 
    & a_{12}^2 - 2a_{12}a_{13} + a_{13}^2 + a_{21}^2 - 2a_{13}a_{23}  - 2a_{21}a_{23} \\
     + & a_{23}^2 - 2a_{21}a_{31}  + a_{31}^2 - 2a_{12}a_{32} - 2a_{31}a_{32} + a_{32}^2) .
\end{align*}
Evaluating at the original vector $\bc$ gives $79$, which is the normalized volume of the original polytope. The volume polynomial $\vol_{\Qa}(\mathbf a)$ for the Euclidean volume of $\Qa$ is given as
\[ \vol_{\Qa}(\mathbf a) = \frac{1}{2} \Vol_{\Qa}(\mathbf a). \]
\end{example}

\begin{remark}
Polytropes have also appeared in the literature as \emph{alcoved polytopes of type $A$}. The volumes of alcoved polytopes of type A were studied in \cite[Theorem 3.2]{LP} and extended to general root systems in \cite[Theorem 8.2]{LP2}, where the normalized volume of an alcoved polytope is described as a sum of discrete volumes of related alcoved subpolytopes. More specifically, given a fixed alcoved polytope of type A, the normalized volume of the respective polytope $P$ can be computed as
\[
    \Vol(P) = \sum_{\omega \in S_{n-1}} |P_\omega\cap \Z^{n-1} |,
\]
where $P_\omega = \{\bx\in\R^{n-1} \mid \bx+\Delta_\omega \subseteq P \}$ and $\Delta_\omega = \{\by \in \R^{n-1} \mid 0 \leq y_{\omega(1)} \leq \dots \leq y_{\omega(n-1)} \leq~ 1\}$.
While this is a formula that yields a value for the normalized volume for polytropes of any dimension, it does not allow a parametrized approach resulting in multivariate polynomials.
\end{remark}

\begin{remark}
The construction described in this section can be applied analogously for certain other classes $\mathcal P$ of polytopes with unimodular facet normals. In this more general framework, as described in \Cref{subsec:toric_geometry}, the ideal $I$ is the toric ideal of the variety $X$, and the ideals $M$ and $L$ are defined accordingly. Suppose the analogue of the matrix $A$ (as described in the paragraph before \Cref{ex:2d-B-description}) fulfills the conditions of \cite{DLS}, i.e. is unimodular and $\ker(A) \cap \RR^n_{\geq 0} = 0$. Then \Cref{alg:cohomology_class} applies and can be used to obtain multivariate volume polynomials of polytopes in the class $\mathcal P$.
\end{remark}

\subsection{Computing multivariate Ehrhart polynomials}\label{subsec:ehr_pols_computation}
We use the Todd operator to pass from the multivariate volume polynomials to the multivariate Ehrhart polynomials of polytropes. We begin by defining single and multivariate versions of the Todd operator and then explain the method we used for computations. Finally, we compute the multivariate and univariate Ehrhart polynomials of our running example. For more thorough background information on the Todd operator, see \cite[Chapter 12]{BR} and \cite[Chapter 13.5]{cox2011toric}.

The Todd operator is related to the Bernoulli numbers, a sequence of rational numbers $B_k$ for $ k \in~ \Z_{\geq 0}$ whose first few terms are $1,-\frac{1}{2},\frac{1}{6},0,-\frac{1}{30},0$.
They are defined through the following generating function:
$$
\frac{z}{\exp (z) -1 } = \sum_{k \geq 0 } \frac{B_k}{k!}z^k.
$$

\begin{definition}
The \emph{Todd operator} is the differential operator
$$ 
\Todd_h = 1 + \sum_{k \geq 1} (-1)^k \frac{B_k}{k!} \left(\frac{d}{dh}\right)^k .
$$
\end{definition}
Note that for a polynomial $f(h)$ of degree $d$, the function $\Todd_h(f)$ is a polynomial: since $(\frac{df}{dh})^k = 0$ for any $k>d$, we get the finite expression
\[
    \Todd_h(f) = 1 + \sum_{k=1}^{d} (-1)^k\frac{B_k}{k!} \left(\frac{df}{dh}\right)^k .
\]
The Todd operator can be succinctly expressed in shorthand as 
$$ \Todd_h =\frac{ \frac{d}{dh}}{1- \exp\left(-\frac{d}{dh}\right)}.$$
In order to compute the multivariate Ehrhart polynomials, we use a 
multivariate version of the Todd operator. For $\mathbf{h} = (h_1, h_2, \dots, h_m)$, we write
$$
\Todd_{\mathbf{h}} = \prod_{j =1 }^{m}\bigg( \frac{ \frac{\partial}{\partial h_j} }{1 -\exp(-\frac{\partial}{\partial h_j})} \bigg).
$$
The Todd operator allows one to pass from a continuous measure of volume on a polytope to a discrete measure: a lattice point count. Let $P = \{\mathbf{x} \in \R ^n :A \mathbf{x} \leq \mathbf{b}\}$, $\mathbf{b} \in \R ^m$. For $ \mathbf{h} \in \R ^m$, the \emph{shifted polytope}
$P_\mathbf{h}$ is defined as 
$$P_\mathbf{h} = \{\mathbf{x} \in \R ^n : A \mathbf{x} \leq \mathbf{b}+\mathbf{h}\}.$$  

\begin{theorem}[Khovanskii-Pukhlikov, {\cite[Ch. 12.4]{BR}}] 
\label{K-Ptheorem}
Let $P \subseteq \R^n$ be a unimodular $d$-polytope. Then
$$ 
\#(P \cap \Z ^n ) = \Todd _ \mathbf{h} \vol(P_\mathbf{h}) |_{\mathbf{h} = 0}.
$$
In words, the number of lattice points of $P$ equals the evaluation of 
the Todd operator at $\mathbf h =0$ on the relative Euclidean volume of the shifted polytope $P_\mathbf h$.
\end{theorem} 
In \Cref{K-Ptheorem}, one applies the Todd operator to the volume of a shifted version $P_\mathbf{h}$ of the polytope $P$. In our setting of multivariate volume polynomials that are constant on fixed cones of the polytrope region in the Gröbner fan, a nice simplification occurs that allows us to ignore this shift. 
As discussed in Section \ref{subsec:background_polytropes}, a polytrope $P$ can be described as
$$P = \{\mathbf{x} \in \R ^{n-1} : x_i - x_j \leq c_{ij},\,  - c_{ni} \leq x_i \leq c_{in} \},$$
where $i, j \in [n-1], i\neq j$, 
for some $\mathbf c \in \R^{n^2-n}.$ Its volume is given by evaluating the multivariate volume polynomial $\vol_P(\mathbf a)$ at $\mathbf c$.
The shifted polytrope $P_\mathbf{h}$ has the description
$$  P_\mathbf{h} = \big \{ \mathbf{x} \in \R^{n-1} : x_i - x_j \leq c_{ij}+h_{ij}, \  -(c_{ni} + h_{ni}) \leq x_i \leq  c_{in} + h_{in} \big \},$$
for any
$\mathbf{h} \in \R^{n^2 - n}$. As long as $\mathbf{h}$ is small enough, the shifted polytrope remains in the same cone and its volume polynomial is given by evaluating the multivariate volume polynomial $\vol_{P}(\mathbf a +\mathbf h)$ at $\mathbf c$. 
As $\vol_P(\mathbf a)$ is a polynomial, 
$$
\left(\prod_{i\ne j\in [n]}\frac{\partial}{\partial h_{ij}}\right)\vol_P(\mathbf a + \mathbf h )\big |_{\mathbf h = 0}
= \left(\prod_{i\ne j \in [n]}\frac{\partial}{\partial a_{ij}}\right) \vol_P(\mathbf a)
.$$
\begin{example}\label{ex:ehr_pol_2d}
We now apply the Todd operator to the multivariate volume polynomial of the tropical hexagon $\Qa$ in our running example. As in the previous example, this $2$-dimensional example can be computed with more elementary methods, such as Pick's formula. However, this example generalizes to higher dimensions, and we use it to present our methods in a manageable size.
Recall from \Cref{ex:vol_pol_2d} that the volume polynomial is
$$
\vol_{\Qa}(\textbf{a}) =\sum_{i \neq j \in [3]}-\frac{1}{2} a_{ij}^2 
+ \sum_{i\neq j \neq k \in [3]} (a_{ij}a_{ik} +a_{ji}a_{ki}).
$$
Evaluating the polynomial at a specific Kleene star $\bc$ returns the volume of the corresponding polytrope.
Applying the multivariate Todd operator to this volume polynomial, we get:
\begin{align*}
\Todd_{\mathbf{h}}\vol_{\Qa}(\mathbf a+\mathbf{h})\bigg |_{\mathbf h = 0}
&= \bigg( \tfrac{ \frac{\partial}{\partial h_{32}} }{1 -\exp(-\frac{\partial}{\partial h_{32}})} \bigg)\dots \bigg( \tfrac{ \frac{\partial}{\partial h_{13}} }{1 -\exp(-\frac{\partial}{\partial h_{13}})} \bigg)\bigg (1+ \sum_{k \geq 1}(-1)^k\frac{B_k}{k!}\big(\frac{\partial}{\partial h_{12}}\big)^k \bigg )\vol_{\Qa}(\mathbf{a} +\mathbf{h} )  \bigg|_{\mathbf h=0}  \\ 
&= \bigg( \tfrac{ \frac{\partial}{\partial a_{32}} }{1 -\exp(-\frac{\partial}{\partial a_{32}})} \bigg)\dots \bigg( \tfrac{ \frac{\partial}{\partial a_{13}} }{1 -\exp(-\frac{\partial}{\partial a_{13}})} \bigg)\bigg (1+ \sum_{k \geq 1}(-1)^k\frac{B_k}{k!}\big(\frac{\partial}{\partial a_{12}}\big)^k \bigg )\vol_{\Qa}(\textbf{a})
\\
&=\bigg( \tfrac{ \frac{\partial}{\partial a_{32}} }{1 -\exp(-\frac{\partial}{\partial a_{32}})} \bigg)\dots \bigg( \tfrac{ \frac{\partial}{\partial a_{13}} }{1 -\exp(-\frac{\partial}{\partial a_{13}})} \bigg)\bigg(\vol_{\Qa}(\textbf{a})+\frac{1}{2}[-a_{12}+a_{13} +a_{32}] -\frac{1}{12}\bigg) \\
& = -\frac{1}{2}a_{12}^2+a_{12}a_{13}-\frac{1}{2}a_{13}^2-\frac{1}{2}a_{21}^2+a_{13}a_{23}+a_{21}a_{23}-\frac{1}{2}a_{23}^2+a_{21}a_{31}-\frac{1}{2}a_{31}^2 \\
& \;\;\;\; +a_{12}a_{32}+a_{31}a_{32}  -\frac{1}{2}a_{32}^2+\frac{1}{2}a_{12}+\frac{1}{2}a_{13}+\frac{1}{2}a_{21}+\frac{1}{2}a_{23}+\frac{1}{2}a_{31}+\frac{1}{2}a_{32}+1 \\
&= \vol_{\Qa}(\textbf{a}) + \sum_{i\neq j \in [3]}\frac{ a_{ij}}{2} + 1.
\end{align*}
Hence, for integral Kleene stars $\bc \in \Z^6$ (i.e. whenever $\Qa[\bc]$ is unimodular), we get that
$$\#(\Qa[\bc] \cap \Z^2) = \vol_{\Qa[\bc]}(\textbf{c}) + \sum_{i\neq j \in [3]}\frac{ c_{ij}}{2} + 1.$$
Note that this implies that $\sum_{i\neq j \in [3]}\frac{ c_{ij}}{2}$ is the number of lattice points on the boundary of $\Qa[c]$.
Evaluating this polynomial at the weight vector $\mathbf c = (3,2,3,4,5,6)$ gives 52, the number of lattice points in the polytrope. Evaluating at $t\mathbf c = (3t,2t,3t,4t,5t,6t)$ gives the univariate Ehrhart polynomial of the polytrope $\Qa[\bc]$:
$$\ehr_{\Qa[c]}(t) = \frac{79}{2}t^2 + \frac{23}{2}t +1 .$$
\end{example}

\subsection{Computing multivariate $h^\ast$-polynomials}\label{subsec:h_pols_computation}
Finally, we can also compute a multivariate 
$h^\ast$-polynomial from the multivariate Ehrhart polynomial 
corresponding to each tropical type. 
We explain the method here. The interested reader can also consult \cite{BR} for further details.

As discussed in \Cref{subsec:ehrhart}, the coefficients $\{ h_i\}$ of the $h^\ast$-polynomial $h^\ast (t) = h_0 +h_1 t + \dots + h_d t^d$ are the coefficients of the Ehrhart polynomial expressed in the basis 
$\left\{ \binom{t+d - i}{d} \mid i \in \{0,1,\dots,d\}\right\}$
of the vector space of polynomials in $t$ of degree at most $d$. To transform the Ehrhart polynomial to the $h^\ast$-polynomial, we perform a change of basis. The Eulerian polynomials play a central role in this transformation.

The Eulerian polynomials $A_d(t)$ are defined through the generating function:
$$\sum_{j\ge 0} j^d t^j = \frac{A_d(t)}{ (1-t)^{d+1}}.$$
 Explicitly, we can write the Eulerian polynomials as
 \[
  A_d(t) = \sum_{m=1}^{d} A(d,m-1) t^m,
 \]
 where $A(d,m)$ is the Eulerian number that counts the number of permutations of $[d]$ with exactly $m$ ascents. The first few Eulerian polynomials are $A_0(t) = 1, A_1(t) = t,$ and $A_2(t) = t^2 +t$.
Recall the Ehrhart series of a $d$-dimensional polytope:
\begin{align*}  
\Ehr_P(t) = \sum_{k \geq 0} \ehr_P(k) t^k  
         = \sum_{k \geq 0}( \lambda_0 + \lambda_1 k + \dots + \lambda_d k^d )t^k 
         = \sum_{i = 0}^{d} \frac{\lambda_i A_i(t) }{(1-t)^{i+1}}.
\end{align*}
On the other hand, we have
$$
\Ehr_P(t) = \frac{ h^\ast_P (t)}{(1-t)^{d+1}}. 
$$
Comparing yields an expression for the $h^\ast$-polynomial in terms of the
coefficients of the Ehrhart polynomial:
$$ 
h^\ast_P(t) = \sum_{i=0}^d \lambda_i A_i(t) (1-t)^{d-i}.
$$
To compute the multivariate $h^\ast$-polynomials, we collect the terms of each degree in the Ehrhart polynomials and apply the transformation.

\begin{example}\label{ex:h*_pol_2d}
    We compute the multivariate $h^*$-polynomial of the hexagon $\Qa$ 
    from the Ehrhart polynomial 
    $\ehr_{\Qa}(t\ba)=\lambda_2 t^2 + \lambda_1 t + 1 $ from \Cref{ex:ehr_pol_2d}. With these coefficients we can compute 
\begin{align*}
    \lambda_2 A_2(t) (1-t)^0 =& \bigg(\sum_{i\neq j \in [3]} -\frac{1}{2} a_{ij}^2 + \sum_{i\neq j \neq k \in [3]}[a_{ij}a_{ik} + a_{ji}a_{ki}] \bigg)(t^2 + t) \\
    \lambda_1 A_1(t) (1-t)^1 =&  \big(\sum_{i\neq j \in [3]}\frac{1}{2}a_{ij}\big)(-t^2 +t)  \\
    \lambda_0 A_0(t) (1-t)^2 =& t^2-2t+1.
\end{align*}
The sum of these three polynomials gives the multivariate $h^*$-polynomial of the hexagon:
\begin{align*}
    h^\ast_{\Qa}(\mathbf a, t) = & \bigg( \sum_{i \neq j \in [3]} -\frac{1}{2}[a_{ij}^2 + a_{ij}]+ \sum_{i\neq j \neq k \in [3]}[a_{ij}a_{ik}+a_{ji}a_{ki}] + 1 \bigg) t^2 \\
    &+ \bigg( \sum_{i \neq j \in [3] }\frac{1}{2} [a_{ij}-a_{ij}^2 ] +\sum_{i \neq j \neq k \in [3]}[a_{ij}a_{ik}+a_{ji}a_{ki}] -2 \bigg) t +1. 
\end{align*} 
Evaluating $h_{\Qa}^\ast(\ba,t)$ at $(\bc,t) = (3,2,3,4,5,6,t)$ 
yields the univariate $h^\ast$-polynomial of the hexagon $\Qa[\bc]$ from \Cref{ex:vol_pol_2d}:
    \[ h^\ast_{\Qa[\bc]} (\mathbf c, t) = 29t^2 + 49t + 1.\]
    The coefficients of $h^*_{\Qa[\bc]}(\mathbf c, t)$ sum to 79, which equals the normalized volume of $\Qa[\bc]$ observed previously in \Cref{subsec:vol_pols_computation}.
\end{example}

\section{Experiments and Observations}
\label{sec:compuatations}
In this section we describe the results of our application of Section \ref{sec:methods} for maximal polytropes of dimension at most 4. Since the Ehrhart and $h^*$-polynomials are computed from the volume polynomials, we mainly focus our investigation on the volume polynomials. All scripts and results of our computations can be found at 
\begin{center}
\url{https://github.com/mariebrandenburg/polynomials-of-polytropes}.
\end{center}

\subsection{Data and computation}

In the computation that is described in this section, we used data from \cite{JS} containing the vertices of one polytrope for each maximal tropical type of dimension $3$ and $4$ up to the action of the symmetric group. The vertices of each polytrope were arranged to form a Kleene star and corresponding weight vector $\bf c$. The methods described in Section \ref{sec:methods} were then applied to obtain multivariate volume, Ehrhart, and $h^*$-polynomials for the corresponding tropical type. Our computations were performed on a desktop computer with a 3.6 GHz quad-core processor. On average, the running time was about 5 minutes for each 4-dimensional volume polynomial, 0.15 seconds for each Ehrhart polynomial, and 0.73 seconds for each $h^*$-polynomial. Parallelization is possible as the computations are independent for each tropical type.

In order to verify our computational results, we independently computed the univariate volume and Ehrhart polynomials with respect to our input data and compared them with our multivariate results, as explained in \Cref{ex:vol_pol_2d,ex:ehr_pol_2d,ex:h*_pol_2d}. To check the $h^*$-polynomial of a representative polytrope, we attempted to compute its $h^*$-polynomial by computing its Ehrhart series with \texttt{Normaliz} and compared this with our multivariate $h^*$-polynomial evaluated at the corresponding weight vector. We attempted to perform this check on a cluster, capping the \texttt{Normaliz} computation of each polytrope's Ehrhart series at 10 minutes. We checked 1459 polytropes. For 670 of them, the \texttt{Normaliz} computation finished in under 10 minutes, and the $h^*$-polynomials matched. Checking the \texttt{Normaliz} computation for individual polytropes revealed that the Ehrhart series computation could take as long as 12 hours, in comparison to the 5 minutes required by our methods. 

\subsection{2-dimensional polytropes}
First we consider $2$-dimensional polytropes. As noted in \Cref{subsec:background_polytropes}, there is a unique class of maximal polytropes up to permutation of vertex labels. The volume, Ehrhart, and $h^*$-polynomials are computed in Examples \ref{ex:vol_pol_2d}, \ref{ex:ehr_pol_2d}, and \ref{ex:h*_pol_2d} respectively. We note that the volume, Ehrhart, and $h^*$-polynomials are all symmetric with respect to the $S_3$ action, as expected.
\subsection{3-dimensional polytropes}
\label{subsec:3D-vols}

In the case of maximal 3-dimensional polytropes, up to the symmetric group action there are 6 types of maximal polytropes. We applied the algorithms in \Cref{sec:methods} to nonnegative points in maximal cones corresponding to these 6 types, yielding the volume, Ehrhart, and $h^*$-polynomials of their corresponding tropical types.

\begin{example}\label{ex:3volpol}
One of the six volume polynomials is

\begin{align*}
&2 a_{12}^3 - 3 a_{12}^2 a_{13} + a_{13}^3 - 3 a_{12}^2 a_{14} + 6 a_{12} a_{13} a_{14} - 3 a_{13}^2 a_{14} + a_{21}^3 - 3 a_{13}^2 a_{23} + 6 a_{13} a_{14} a_{23} - 3 a_{14}^2 a_{23} \\
&- 3 a_{14} a_{23}^2 - 3 a_{21} a_{23}^2 + a_{23}^3 - 3 a_{21}^2 
 a_{24} + 6 a_{14} a_{23} a_{24} + 6 a_{21} a_{23} a_{24} - 3 a_{14} a_{24}^2 - 3 a_{23} a_{24}^2 + a_{24}^3 - 3 a_{21}^2a_{31}\\ 
& + 6 a_{21} a_{24} a_{31} - 3 a_{24}^2 a_{31} - 3 a_{24} a_{31}^2 + a_{31}^3 - 3 a_{12}^2 a_{32} + 6 a_{12} a_{14} a_{32} 
 - 3 a_{14}^2 a_{32} - 3 a_{31}^2 a_{32} - 3 a_{14} a_{32}^2 \\
& + 6  a_{14} a_{24} a_{34} 
 + 6 a_{24} a_{31} a_{34} + 6 a_{14} a_{32} a_{34} + 6 a_{31} a_{32} a_{34} - 3 a_{14} a_{34}^2 - 3 a_{24} a_{34}^2 - 3 a_{31} a_{34}^2 - 3 a_{32} a_{34}^2 
 + 2 a_{34}^3 \\
 & + 6 a_{21} a_{31} a_{41} - 3 a_{31}^2 a_{41} 
 + 6 a_{31} a_{32} a_{41} - 3 a_{32}^2 a_{41} - 3 a_{21} a_{41}^2 - 3 a_{32} a_{41}^2 + a_{41}^3 - 3 a_{12}^2 a_{42} + 6 a_{12} a_{13} a_{42} \\
& - 3 a_{13}^2 a_{42} + 6 a_{12} a_{32} a_{42} + 6 a_{32} a_{41} a_{42} - 3 a_{13} a_{42}^2 - 3 a_{32} a_{42}^2 - 3 a_{41} a_{42}^2 + a_{42}^3 - 3 a_{21}^2 a_{43} + 6 a_{13} a_{23} a_{43} \\
& + 6 a_{21} a_{23} a_{43}   - 3 a_{23}^2 a_{43} + 6 a_{21} a_{41} a_{43} - 3 a_{41}^2 a_{43} + 6 a_{13} a_{42} a_{43} + 6 a_{41} a_{42} a_{43} 
 - 3 a_{13} a_{43}^2 - 3 a_{21} a_{43}^2 \\ 
 &- 3 a_{42} a_{43}^2 + a_{43}^3.
\end{align*}
\end{example}

We devote the remainder of this subsection to an analysis of the coefficients of the normalized volume polynomials, which we write as follows:
\[ \Vol(\{\mathbf x\in \RR^4 \mid x_i-x_j \le a_{ij}, x_4=0\}) = \sum_{\mathbf v} \alpha_{\mathbf v} {\bf a}^{\mathbf v},\]
where $\mathbf v\in \mathbb N^{12}$ has coordinates summing to 3. Note that there is a natural decomposition of the set of all possible exponent vectors $\mathbf v$ into three different disjoint subsets $T_{111}, T_{21},$ and $T_3$, one for each partition of 3. 

Recall that the 6 types of maximal 3-dimensional polytropes correspond to different regular central triangulations of the fundamental polytope $FP_4$, as discussed in \Cref{subsec:background_polytropes}. A regular central triangulation is determined by a choice of triangulating edge in each of the six square facets of $FP_4$. The coefficients of the volume polynomials encode the data of these six facet triangulations as follows:

\begin{itemize}
    \item Let $\mathbf v\in T_{111}$, so that the monomial ${\bf a}^{\mathbf v}$ is $a_{ij} a_{kl} a_{st}$ for some $i\ne j, k\ne l, s \ne t$ and $(i,j)\ne (k,l) \ne (s,t)$. The coefficients $\alpha_{\bv}$ in this case are determined directly by the triangulation of $FP_4$: 
    \[\alpha_{\bv}=\begin{cases}
        6 & \text{if }\be_i-\be_j, \be_k-\be_l, \be_s-\be_t\text{ form a 2-dimensional simplex in the } \\
          & \text{corresponding regular central triangulation,}\\
        0 & \text{ otherwise.}
    \end{cases}\]
    \item Let $\bv\in T_{21}$, so that the monomial ${\mathbf a}^\bv$ is $a_{ij}^2 a_{kl}$ for some $i\ne j$, $k\ne l$, and $(i,j)\ne (k,l)$. The coefficient $\alpha_{\bv}$ is nonzero only if $\be_i-\be_j$ and $\be_k-\be_l$ are adjacent vertices of $FP_4$. In that case, it is determined by the square facet $S$ of $FP_4$ containing $\be_i-\be_j$ and $\be_k-\be_l$:
    \[\alpha_{\bv}=\begin{cases}
        -3 & \text{if }\be_k-\be_l\text{ incident to triangulating edge of }S\\
        0 & \text{otherwise.}
    \end{cases}\] 
    \item Let $\bv\in T_3$, so that the monomial ${\bf a}^\bv$ is $a_{ij}^3$ for some $i\ne j$. The coefficient $\alpha_{\bv}$ is given by
    \[\alpha_{\bv}= {7-\deg(\be_i-\be_j)},\]
    where $\deg(\be_i-\be_j)$ is the number of edges incident to the vertex $\be_i-\be_j$ in the regular central subdivision of $FP_4$.
    \end{itemize}
We note that the above descriptions of the coefficients of the volume polynomial imply that the sums of coefficients corresponding to each partition of $3$ are the same for all six volume polynomials:
\[\sum_{\bv\in T_3} \alpha_{\bv} = 12,\ \sum_{\bv\in T_{21}} \alpha_{\bv} = -108,\ \sum_{\bv\in T_{111}} \alpha_{\bv} = 120.\]
\begin{example}
\label{ex:3d-polytrope-volume}
    Consider the polytrope $P$ with facet coefficients $c_{ij}$ given by the matrix
    \[\begin{pmatrix}
    0 & 11 & 20 & 29\\
    21 & 0 & 19 & 20\\
    20 & 29 & 0 & 11\\
    19 & 20 & 21 & 0
    \end{pmatrix}.\]
    Assigning the weight $c_{ij}$ to the vertex $\be_i-\be_j$ of the fundamental polytope $FP_4$, and weight 0 to the central vertex at the origin, produces the regular central triangulation in \Cref{fig:3d-polytrope-subdiv}. 
    The volume polynomial corresponding to this polytrope is the polynomial displayed in \Cref{ex:3volpol}. We see that the coefficients corresponding to $a_{12}^3, a_{12}^2a_{14}, a_{32}^2a_{42},$ and $a_{31}a_{32}a_{41}$ are equal to $2, -3, 0,$ and $6$ respectively, as summarized by the discussion above.
    \begin{figure}[!ht]
        \centering
        \includegraphics[height=3in]{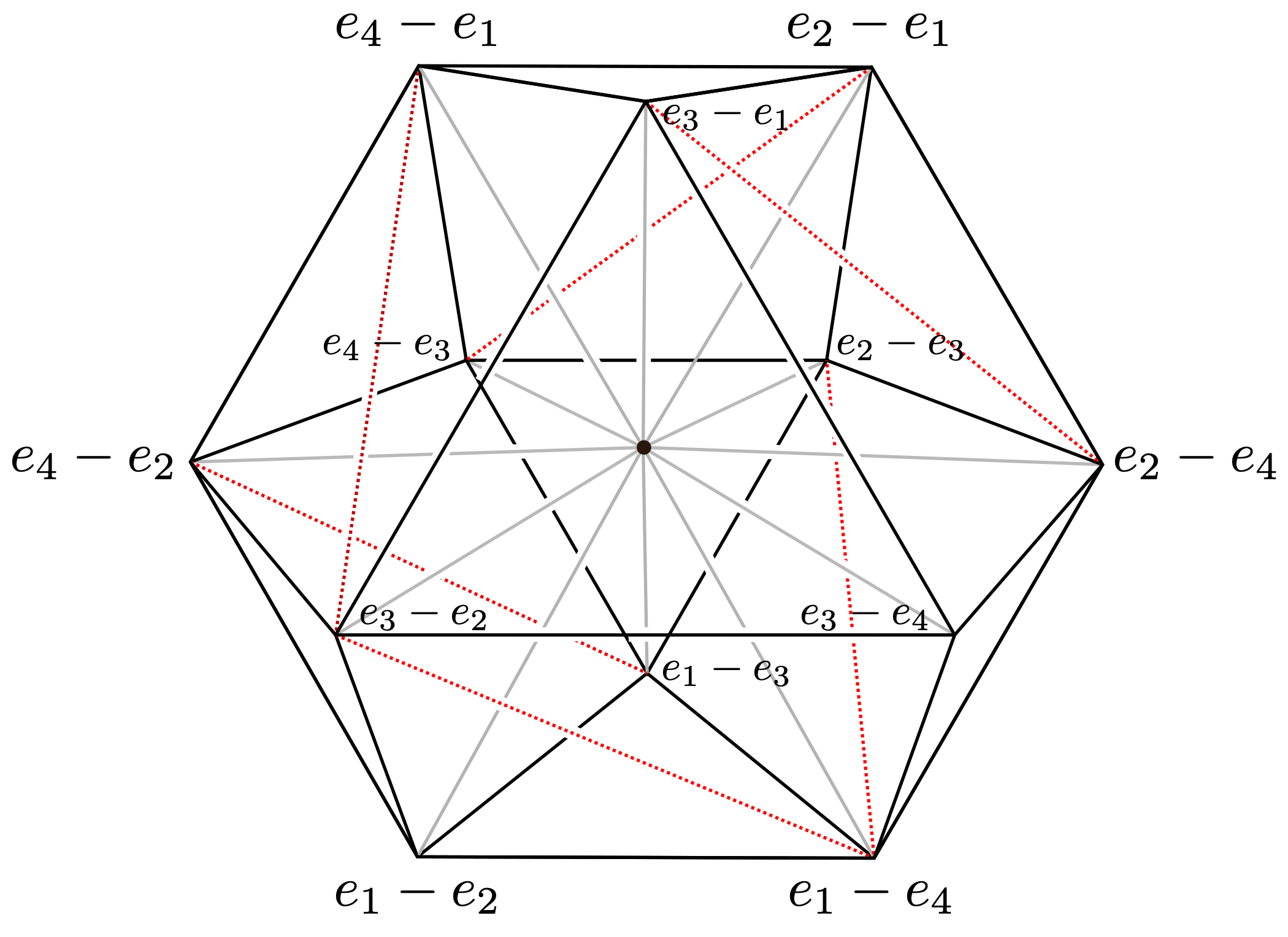}
        \caption{The regular central triangulation of $FP_4$ corresponding to the polytrope in \Cref{ex:3d-polytrope-volume}, with triangulating edges of square facets of $FP_4$ colored red and dashed.}
        \label{fig:3d-polytrope-subdiv}
    \end{figure}
\end{example}

\subsection{4-dimensional polytropes}
\label{subsec:4D-vols}
Finally we consider 4-dimensional polytropes. In this case, up to the action of the symmetric group $S_5$ there are 27248 types of maximal polytropes. We applied the methods of \Cref{sec:methods} to obtain multivariate volume, Ehrhart, and $h^*$-polynomials for these polytropes.

We can embed the 27248 normalized volume polynomials using the canonical basis in the vector space of homogeneous polynomials of degree 4, having dimension $\binom{23}{4}=8855$. The affine span of these volume polynomials has dimension 70, implying that there is much structure in their coefficients. We note that this equals the number of facets in a regular central triangulation of $FP_5$. 

We were able to experimentally verify the facts collected in \Cref{tab:4d-stats}. For example, all coefficients for monomials corresponding to the partition $2 + 2 = 4$ lie in the set $\{0,6\}$, and the sum of all such coefficients is $300$.
Furthermore, the $S_5$-orbit of the monomials $a_{12}a_{13}a_{14}a_{15}$ and $a_{21}a_{31}a_{41}a_{51}$ always appears in the volume polynomial with coefficient $24$. Finally, the coefficient $-4$ always appears exactly twice as often as the coefficient $12$. 
\begin{table}[h]
\begin{center}
\begin{tabular}{ |c|c|c|c| } 
 \hline
Partition & Example monomial & Possible coefficients & Coefficient sum \\ [0.5ex] 
 \hline 
 4       & $a_{12}^4$                 & $-6,-3,-2,-1,0,1,2,3$ & $-20$   \\ [0.5ex] 
 3 + 1   & $a_{12}^3a_{13}$           & $-4,0,4,8$            & $320$  \\ [0.5ex] 
 2 + 2   & $a_{12}^2a_{13}^2$         & $0,6$                 & $300$  \\ [0.5ex] 
 2+1+1   & $a_{12}a_{13}a_{14}^2$     & $-12,0,12$            & $-2160$\\ [0.5ex] 
 1+1+1+1 & $a_{12}a_{13}a_{14}a_{15}$ & $0,24$                & $1680$ \\ [0.5ex] 
 \hline
\end{tabular}
\end{center}
\caption{Summary statistics for coefficients of 4-dimensional volume polynomials.}
\label{tab:4d-stats}
\end{table}

As in the 3-dimensional case, a monomial corresponding to the partition $1+1+1+1=4$ had coefficient $24$ if and only if it appeared as a face in the corresponding triangulation. Beyond these observations, we were unable to detail the exact relationship between the volume polynomials and their corresponding regular central triangulations.
\begin{question}
How do the coefficients of the volume polynomials of maximal $(n-1)$-dimensional polytropes reflect the combinatorics of the corresponding regular central subdivision of $FP_n$?
\end{question}
A natural first step would be to prove that, for $\bv$ with partition $1+1+\dots+1=n-1$, the coefficient $\alpha_{\bv}$ is nonzero if and only if it corresponds to a face in the regular central triangulation.

\printbibliography
\end{document}